\documentclass[a4paper,12pt]{article}

\usepackage[margin=2cm]{geometry}
\usepackage{amsmath,amsthm,amssymb,latexsym}

\theoremstyle{plain}
\newtheorem{proposition}{Proposition}[section]
\newtheorem{theorem}[proposition]{Theorem}
\newtheorem{corollary}[proposition]{Corollary}
\newtheorem{lemma}[proposition]{Lemma}
\theoremstyle{definition}
\newtheorem{definition}[proposition]{Definition}
\newtheorem{example}[proposition]{Example}
\newtheorem{remark}[proposition]{Remark}

\newcommand{\mc}{\mathcal}
\newcommand{\mf}{\mathfrak}
\newcommand{\ip}[2]{\langle{#1},{#2}\rangle}
\newcommand{\G}{\mathbb G}
\newcommand{\vnten}{\overline\otimes}
\newcommand{\im}{\operatorname{im}}
\newcommand{\re}{\operatorname{re}}
\newcommand{\lin}{\operatorname{lin}}
\newcommand{\id}{\operatorname{id}}
\newcommand{\hh}{\widehat}
\newcommand{\proten}{\widehat\otimes}

\begin{document}

\title{One-parameter isometry groups and inclusions between operator algebras}
\author{Matthew Daws}
\maketitle

\begin{abstract}
We make a careful study of one-parameter isometry groups on Banach spaces, and their
associated analytic generators, as first studied by Cioranescu and Zsido.  We pay
particular attention to various, subtly different, constructions which have appeared in
the literature, and check that all give the same notion of generator.  We give an
exposition of the ``smearing'' technique, checking that ideas of Masuda, Nakagami and
Woronowicz hold also in the weak$^*$-setting.  We are primarily interested in the case of
one-parameter automorphism groups of operator algebras, and we present many applications
of the machinery, making the argument that taking a structured, abstract approach can pay
dividends.  A motivating example is the scaling group of a locally compact quantum group
$\G$ and the fact that the inclusion $C_0(\G) \rightarrow L^\infty(\G)$ intertwines the
relevant scaling groups.  Under this general setup, of an inclusion of a $C^*$-algebra
into a von Neumann algebra intertwining automorphism groups, we show that the graphs of
the analytic generators, despite being only non-self-adjoint operator algebras,
satisfy a Kaplansky Density style result.  The dual picture is the inclusion
$L^1(\G)\rightarrow M(\G)$, and we prove an ``automatic normality'' result under this
general setup.  The Kaplansky Density result proves more elusive, as does a general study
of quotient spaces, but we make progress under additional hypotheses.
\end{abstract}

\section{Introduction}

A one-parameter automorphism group of an operator algebra is $(\alpha_t)_{t\in\mathbb R}$
where each $\alpha_t$ is an automorphism, we have the group law $\alpha_t \circ \alpha_s
= \alpha_{t+s}$, and a continuity condition on the orbit maps $a\mapsto\alpha_t(a)$
(either norm continuity for a $C^*$-algebra, or weak$^*$-continuity for a von Neumann
algebra).  As for the more common notion of a semigroup of operators, such groups admit
a ``generator'', an in general unbounded operator which characterises the group.
This paper will be concerned with the \emph{analytic generator}, formed by complex
analytic techniques, which can loosely be thought of as the exponential of the more
common infinitesimal generator.

The analytic generator was defined and studied in \cite{cz}, see also \cite{z,z2,z1},
\cite[Appendix~F]{mnw}, \cite{kus1}.  There are immediate links with Tomita-Takesaki
theory, \cite[Chapter~VIII]{tak2} and \cite{z1}, although we contrast the explicit use
of generators in \cite{z1} with the more adhoc approach of \cite{tak2}.  Our principle
interest comes from the operator algebraic approach to quantum groups,
\cite{kv}, and specifically the treatment of the antipode.  For a quantum
group, the antipode represents the group inverse, and is represented as an, in general
unbounded, operator $S$ on an operator algebra.  This operator factorises as
$S=R\tau_{-i/2}$ where $R$ is the \emph{unitary antipode}, an anti-$*$-homomorphism, and
$\tau_{-i/2}$ which is an analytic continuation of a one-parameter automorphism group, the
\emph{scaling group} $(\tau_t)$.  Furthermore, $S^2=\tau_{-i}$ which is precisely the
analytic generator.

We tend to think of the quantum group
$\G$ as an ``abstract object'' which can be represented be a variety of operator algebras,
in particular the reduced $C^*$-algebra $C_0(\G)$, thought of as functions on $\G$
vanishing at infinity, and the von Neumann algebra $L^\infty(\G)$, thought of as
measurable functions on $\G$.  There is a natural inclusion $C_0(\G)\rightarrow
L^\infty(\G)$, which intertwines the scaling group(s)--- the scaling group is
norm-continuous on $C_0(\G)$ and weak$^*$-continuous on $L^\infty(\G)$.  Much of this
paper is concerned with this situation in the abstract: an inclusion of a $C^*$-algebra
into a von Neumann algebra which intertwines automorphism groups.  Such a situation also
occurs in Tomita-Takesaki theory, where a convenient way to construct type~III von
Neumann algebras is to start with a KMS state on $C^*$-algebra and to apply the GNS
construction, see \cite{i} for example.
One of our main results,
Theorem~\ref{thm:six}, gives a Kaplansky density result for the graphs of the analytic
generators in such a setting.

Using the coproduct we can turn the dual spaces into Banach algebras.  This leads to the
dual of $C_0(\G)$, denoted $M(\G)$ and thought of as a convolution algebra of measures,
and also to the predual of $L^\infty(\G)$, denoted $L^1(\G)$ and thought of as the
absolutely continuous measures.  These do not carry a natural involution, because we would
wish to use the antipode which is not everywhere defined, but there are natural dense
$*$-subalgebras, $L^1_\sharp(\G)$ and $M_\sharp(\G)$, compare Section~\ref{sec:lcqgs}
below.  Part of our motivation for writing this paper was to attempt to understand our
result, with Salmi, that when $\G$ is coamenable, there is a Kaplansky density result for
the inclusion $L^1_\sharp(\G) \rightarrow M_\sharp(\G)$; compare Proposition~\ref{prop:15}
below, where we are still unable to remove the coamenability condition.  A positive
general result is Theorem~\ref{thm:10} which shows that if $\omega\in L^1(\G)$ and
$\omega^*\circ S$ is bounded on $D(S)\subseteq L^\infty(\G)$, then
$\omega\in L^1_\sharp(\G)$.  This is notable because it gives a criterion to be a member
of $L^1_\sharp(\G)$ which is not ``graph-like'': we do not suppose the existence of
another member of $L^1(\G)$ interacting with $S$ in some way.

A further motivation for writing this paper was to make the case that considering the
analytic generator (or rather, the process of analytic continuation) as a theory in its
own right has utility; compare with the adhoc approach of \cite{tak2} or \cite{vd}.
In particular, we take a great deal of care to consider the various different topologies
that have been used in the literature, and to verify that these lead to the same
constructions:
\begin{itemize}
\item Either the weak, or norm, topology gives the same continuity assumption on
the group $(\alpha_t)$ (this is well-known) but it is not completely clear that norm
analytic continuation (as used in \cite{mnw} for example) is the same as weak
analytic continuation (which is the framework of \cite{cz}).  Theorem~\ref{thm:one}
below in particular implies that it is.
\item For a von Neumann algebra, \cite{cz} used weak$^*$-continuity, but it is also
common to consider the $\sigma$-strong$^*$ topology, \cite{kus2,kus3}, or the strong
topology, \cite{fv} for example.  A priori, it is hence not possible to apply the
results of \cite{cz} (for example) to the definition used in \cite{kus2}.
Theorem~\ref{thm:10a} below shows that these do however give the same analytic extensions.
\item It is also possible to use duality directly; this approach is taken in \cite{vd}
for example.  Duality is explored in \cite{z}; compare Theorem~\ref{thm:two} below.
\end{itemize}

In Section~\ref{sec:1pgps} we give an introduction to one-parameter isometry groups on
Banach spaces and explore and prove the topological results summarised above.
We also explore some examples.  Section~\ref{sec:smearing} is devoted to the technique
of ``smearing'', and in particular to the ideas of \cite[Appendix~F]{mnw}, which we
find to be very powerful.  We check that the ideas of \cite[Appendix~F]{mnw} also
work for weak$^*$-continuous groups.  These first two sections are deliberately
expositionary in nature.

In Section~\ref{sec:apps} we present a variety of applications of the smearing technique.
We give new proofs of some known results (for example, Zsido's result that the graph of
the generator is an algebra, without using the machinery of spectral subspaces).
In the direction of Tomita-Takesaki theory, as an example of the utility of taking a
structured approach, we show how the main result of \cite{bcm}
follows almost immediately from the work of Cioranescu and Zsido in \cite{cz}, and give
another application of smearing to prove the remainder the results of \cite{bcm}.
We finish by making some remarks on considering the graph of the generator as a Banach
algebra: we believe there is interesting further work here.

In Section~\ref{sec:kap} we formulate and prove a Kaplansky Density result.
Given a $C^*$-algebra $A$ included in a von Neumann algebra $M$ with $A$ generating $M$,
Kaplansky Density says that the unit ball of $A$ is weak$^*$-dense in the unit ball of
$M$.  If $(\alpha_t)$ is an automorphism group of $M$ which restricts to a norm-continuous
group on $A$, then we can consider the graphs of the generators, say
$\mc G(\alpha^A_{-i})$ and $\mc G(\alpha^M_{-i})$, which are non-self-adjoint operator
algebras.  We have that $\mc G(\alpha^A_{-i}) \subseteq \mc G(\alpha^M_{-i})$ and is
weak$^*$-dense (see Proposition~\ref{prop:three} for example).  The main result here is
that the unit ball of $\mc G(\alpha^A_{-i})$ is weak$^*$-dense in the unit ball of
$\mc G(\alpha^M_{-i})$.  The key idea is to consider the bidual
$\mc G(\alpha^A_{-i})^{**}$, and to identify $\mc G(\alpha^M_{-i})$ within this.

In Section~\ref{sec:duals}, we consider the ``adjoint'' of the above situation, the
inclusion $M_*\rightarrow A^*$.  Our groundwork in Section~\ref{sec:kap} leads us to
show Theorem~\ref{thm:seven} which shows that if $\omega\in M_*$ and $\omega\in
D(\alpha_{-i}^{A^*})$ then automatically $\alpha_{-i}^{A^*}(\omega) \in M_*$, so
that $\omega\in D(\alpha_{-i}^{M_*})$.  The analogous result for the inclusion
$A\rightarrow M$ is false, see Example~\ref{ex:four}.
We make a study of quotients.  For both dual spaces, and quotients, we
seem to require extra hypotheses (essentially, forms of complementation).  We finish
by making some remarks about ``implemented'' automorphism groups, as studied further
in \cite[Section~6]{cz} and \cite{z1}.  In the final section we apply our results to
the study of locally compact quantum groups.

\subsection{Notation}

We use $E,F$ for Banach spaces, and write $E^*$ for the dual space of $E$.
For $x\in E, \mu\in E^*$ we write $\ip{\mu}{x} = \mu(x)$ for the pairing.
Given a bounded linear map $T:E\rightarrow F$ we write $T^*$ for the (Banach space)
adjoint $T^*:F^*\rightarrow E^*$.  This should not cause confusion with the Hilbert space
adjoint.  We use $A$ for a Banach or $C^*$-algebra, and $M$ for a von Neumann algebra,
writing $M_*$ for the predual of $M$.  

If $E_0\subseteq E$ is a closed subspace, then by the
Hahn-Banach theorem we may identify the dual of $E_0$ with $E^* / E_0^\perp$, and
identify $(E/E_0)^*$ with $E_0^\perp$, where
\[ E_0^\perp = \{ \mu\in E^* : \ip{\mu}{x}=0 \ (x\in E_0) \}. \]
Similarly, for a subspace $X\subseteq M$ we define ${}^\perp X = \{ \omega\in M_* :
\ip{x}{\omega}=0 \ (x\in X) \}$.  The weak$^*$-closure of $X$ is $({}^\perp X)^\perp$,
and if $X$ is weak$^*$-closed, then $M_* / {}^\perp X$ is the canonical predual of $X$.

By a \emph{metric surjective} $T:E\rightarrow F$ we mean a surjective bounded linear
map such that the induced isomorphism $E/\ker T\rightarrow F$ is an isometric isomorphism.
By Hahn-Banach, this is if and only if $T^*:F^*\rightarrow E^*$ is an isometry onto
its range (which is $(\ker T)^\perp$).

\subsection{Acknowledgements}

The author would like to thank Thomas Ransford, Piotr So{\l}tan, and Ami Viselter for
helpful comments and careful reading of a preprint of this paper, as well as the anonymous
referee for their helpful comments.

\section{One-parameter groups}\label{sec:1pgps}

A \emph{one-parameter group of isometries} on a Banach space $E$ is a family
$(\alpha_t)_{t\in\mathbb R}$ of bounded linear operators on $E$ such that $\alpha_0$ is the
identity, each $\alpha_t$ is a contraction, and $\alpha_t \circ \alpha_s = \alpha_{t+s}$
for $s,t\in\mathbb R$.  Then $\alpha_{-t}$ is the inverse to $\alpha_t$, and thus each
$\alpha_t$ is actually an isometric isomorphism of $E$.

We want to consider one of a number of continuity conditions on $(\alpha_t)$:
\begin{enumerate}
\item We say that $(\alpha_t)$ is \emph{norm-continuous} if, for each $x\in E$, the orbit map
$\mathbb R\rightarrow E; t\mapsto \alpha_t(x)$ is continuous, for the norm topology on $E$;
\item We say that $(\alpha_t)$ is \emph{weakly-continuous} if each orbit map is continuous
for the weak topology on $E$.  However, this condition implies already that $(\alpha_t)$ is
norm-continuous; see \cite[Proposition~1.2']{tak2} for a short proof.
\item If $E$ is the dual of a Banach space $E_*$, then $(\alpha_t)$ is
\emph{weak$^*$-continuous} if each operator $\alpha_t$ is weak$^*$-continuous,
and the orbit maps are weak$^*$-continuous.
\end{enumerate}

\begin{example}\label{ex:one}
Consider the Banach spaces $c_0(\mathbb Z)$ and $\ell^\infty(\mathbb Z)$.
Let $\alpha_t$ be the operator given by multiplication by $(e^{int})_{n\in\mathbb Z}$.
Then $(\alpha_t)$ forms a one-parameter group of isometries which is norm-continuous
on $c_0(\mathbb Z)$, and which is weak$^*$-continuous on $\ell^\infty(\mathbb Z)$, but not
norm-continuous on $\ell^\infty(\mathbb Z)$ (consider the orbit of the constant sequence
$(1) \in\ell^\infty(\mathbb Z)$).
\end{example}

We shall mainly be interested in the case of a Banach algebra $A$.  If each $(\alpha_t)$
is an algebra homomorphism, then we call $(\alpha_t)$ a \emph{(one-parameter) automorphism
group}.  If $A$ is a $C^*$-algebra, then we require that each $\alpha_t$ be a
$*$-homomorphism, and, unless otherwise specified, we suppose that $(\alpha_t)$ is
norm-continuous.  When $A=M$ is actually a von Neumann algebra, unless otherwise
specified, we assume that $(\alpha_t)$ is weak$^*$-continuous.  When $M$ acts on a Hilbert
space $H$, there are of course other natural topologies on $M$, and we shall make some
comments about these later, see Theorem~\ref{thm:10a} below, for example.

In the classical theory of, say, $C_0$-semigroups (where we replace $\mathbb R$ by
$[0,\infty)$) central to the theory is the notion of a generator.  This paper will be
concerned with a different idea, the \emph{analytic generator}, which arises from complex
analysis techniques.  Here we follow \cite{cz}; see also \cite{kus1} in the norm-continuous
case, and the lecture notes \cite[Section~5.3]{kus2}.

\begin{definition}
For $z\in\mathbb C\setminus\mathbb R$ define 
\[ S(z) = \{ w\in\mathbb C : 0 \leq \im w / \im z \leq 1 \}. \]
That is, $S(z)$ is the closed horizontal strip bounded by $\mathbb R$ and $\mathbb R+z$.
For $t\in\mathbb R$ let $S(t)=\mathbb R$.

For a Banach space $E$, a function $f:S(z)\rightarrow E$ is \emph{norm-regular} when $f$
is continuous, and analytic in the interior of $S(z)$.
\end{definition}

Notice that we make no boundedness assumption, but see Remark~\ref{rem:2} below.

We remind the reader that for a domain $U\subseteq\mathbb C$ and $f:U\rightarrow E$,
we have that $f$ is analytic (in the sense of having an absolutely convergent power series,
locally to any point in $U$) if and only if $\mu\circ f$ is complex differentiable, for
each $\mu\in E^*$.  If $E=(E_*)^*$ is a dual space, then it suffices that $f$ be
``weak$^*$-differentiable'', that is, we test only for $\mu\in E_*$.  For a short proof see
\cite[Appendix~A1]{tak2}, and for further details, see for example \cite{an, ane}.

When $E=(E_*)^*$ is a dual space, we say that $f:S(z)\rightarrow E$ is \emph{weak$^*$-regular}
when $f$ is weak$^*$-continuous.  By the above remarks, it does not matter which notion of
``analytic'' we consider on the interior of $S(z)$.

\begin{definition}
Let $(\alpha_t)$ be a norm-continuous, one-parameter group of isometries on $E$, and let
$z\in\mathbb C$.  Define a subset $D(\alpha_z)\subseteq E$ by saying
that $x\in D(\alpha_z)$ when there is a norm-regular $f:S(z)\rightarrow E$ with
$f(t)=\alpha_t(x)$ for each $t\in\mathbb R$; in this case, we set $\alpha_z(x) = f(z)$.

We make the same definition for a weak$^*$-continuous isometry group, using a
weak$^*$-regular map $f$.
\end{definition}

Suppose we have two regular maps $f,g:S(z)\rightarrow E$ with $f(t) = g(t)=\alpha_t(x)$
for each $t\in\mathbb R$.
For $\mu\in E^*$ (or $E_*$ in the weak$^*$-continuous case) consider the map
$h:S(z)\rightarrow\mathbb C; w\mapsto \ip{\mu}{f(w)-g(w)}$.  Then $h$ is regular and
vanishes on $\mathbb R$, and so by the reflection principle, and Morera's Theorem, we
can extend $h$ to an analytic function on the interior of $S(z)\cup S(-z)$ which vanishes
on $\mathbb R$, and which hence vanishes on all of $S(z)$.  As $\mu$ was arbitrary, this
shows that $f(w)=g(w)$ for each $w\in S(z)$.  We conclude that the regular map occurring
in the definition of $\alpha_z$ is unique; we term $f$ an analytic extension of the orbit
map $t\mapsto \alpha_t(x)$.

It is easy to show that $D(\alpha_z)$ is a subspace of $E$, and
that $\alpha_z: D(\alpha_z)\rightarrow E$ is a linear operator.
We remark that \cite{cz} uses a vertical strip instead, but one can simply ``rotate''
the results to our convention.
We have the familiar properties (see \cite[Section~1]{kus1}, \cite[Section~2]{cz}), all
of which follow essentially immediately from uniqueness of analytic extensions:
\begin{enumerate}
\item $\alpha_t \circ \alpha_z = \alpha_z\circ\alpha_z = \alpha_{z+t}$ for
$t\in\mathbb R$; here using the usual notion of composition of not necessarily everywhere
defined operators.
\item if $w\in S(z)$ then $\alpha_z \subseteq \alpha_w$.  It follows that
$S(z)\rightarrow E; w\mapsto \alpha_w(x)$ is defined, and by uniqueness, is the
analytic extension of the orbit map for $x$.
\item $\alpha_{-z} = \alpha_z^{-1}$.
\item $\alpha_{z_1} \circ \alpha_{z_2} \subseteq \alpha_{z_1+z_2}$, with
equality if both $z_1,z_2$ lie on the same side of the real axis.
\end{enumerate}
Furthermore, $\alpha_z$ is a closed operator (see  \cite[Theorem~1.20]{kus1} for the
norm-continuous case, and \cite[Theorem~2.4]{cz} for the weak$^*$-continuous case).

\begin{remark}\label{rem:2}
Contrary to some sources, we have not imposed any boundedness assumptions on our
regular maps; however, in our setting, this is automatic.  Let $z=t+is\in\mathbb C$ and
$x\in D(\alpha_z)$.  Then $x\in D(\alpha_{is})$ and $\alpha_z(x) =
\alpha_t(\alpha_{is}(x))$ and so $\|\alpha_z(x)\| = \|\alpha_{is}(x)\|$.  In the
rest of this remark, we will assume without loss of generality that $s>0$.

In the norm-continuous case, the map $[0,s]\rightarrow E; r\mapsto \alpha_{ir}(x)$
is norm-continuous, and so has bounded image.  As $(\alpha_t)$ is an isometry group,
it follows that $w\mapsto \alpha_w(x)$ is bounded on $S(z)$.  By the Three-Lines Theorem,
if we set
\[ M = \max\Big( \sup_r \|\alpha_r(x)\|, \sup_r \|\alpha_{is+r}(x)\| \Big)
= \max(\|x\|, \|\alpha_z(x)\|), \]
then $\|\alpha_w(x)\| \leq M$ for each $w\in S(z)$.

In the weak$^*$-continuous\footnote{See Appendix~\ref{add:1} below for clarification.} case, for any $\mu\in E_*$, the map $[0,s]\rightarrow
\mathbb C; r\mapsto \ip{\alpha_{ir}(x)}{\mu}$ is continuous and so bounded, and so,
again, the Three-Lines Theorem shows that $|\ip{\alpha_w(x)}{\mu}| \leq M\|\mu\|$ for
$w\in S(z)$.  Taking the supremum over $\|\mu\|\leq1$ shows that $\|\alpha_w(x)\|
\leq M$ for $w\in S(z)$.

Similar remarks would also apply to weakly-continuous extensions, if we were to consider
these.
\end{remark}

The paper \cite{cz} works with general dual pairs of Banach spaces, which satisfy certain
axioms.  In particular, if $(\alpha_t)$ is norm-continuous on $E$, then it is
weakly-continuous, and so we can consider weakly-regular extensions, to which the general
theory of \cite{cz} applies.

\begin{remark}\label{rem:five}
In particular, the dual pairs of Banach spaces which \cite{cz} considers admit a ``good''
integration theory.  We shall only consider the cases of weak$^*$-continuous maps, for
which we can just consider weak$^*$-integrals; and weakly-continuous maps, for which the
theory is less obvious.  Indeed, let $f:\mathbb R\rightarrow E$ be weakly continuous
with $\int_{\mathbb R}\|f(t)\| \ dt<\infty$.  A naive definition of $\int_{\mathbb R}
f(t) \ dt$ defines a member of $E^{**}$, but this integral actually converges in $E$,
see \cite[Proposition~1.4]{cz} and \cite[Proposition~1.2]{arv1}.  Alternatively, if $E$
is separable, we can use the Bochner integral and the Pettis Measurability Theorem.
\end{remark}

Suppose $x\in E$ and $f:S(z)\rightarrow E$ is a weakly-regular extension of the orbit map
for $x$.  Then $t\mapsto f(t)=\alpha_t(x)$ is norm-continuous, and also $t\mapsto f(t+z)
= \alpha_t(\alpha_z(x)) = \alpha_t(f(z))$ (by property (1) above) is norm-continuous.
Further, on the interior of $S(z)$, we have that $f$ is analytic, and hence
norm-continuous.  However, it is not immediately clear why $f$ need be norm-continuous on
all of $S(z)$.  We now show that actually $f$ is automatically norm-continuous; but below
we give an example to show that under slightly weaker conditions, norm-continuity on all
of $S(z)$ can fail, showing that this is more subtle than it might appear.

\begin{theorem}\label{thm:one}
Let $E$ be a Banach space, and let $f:S(z)\rightarrow E$ be a bounded, weakly-regular map.
Assume further that $t\mapsto f(t)$ and $t\mapsto f(z+t)$ are norm continuous.
Then $f$ is norm-regular.
\end{theorem}
\begin{proof}
Define $g:S(z)\rightarrow E$ by $g(w) = e^{-w^2}f(w)$.  Then $g$ is weakly-regular,
and $t\mapsto g(t)$ and $t\mapsto g(z+t)$ are uniformly (norm) continuous.

We now use a ``smearing'' technique.  For $n>0$ define $g_n:S(z)\rightarrow E$ by
\[ g_n(w) = \frac{n}{\sqrt\pi} \int_{\mathbb R} e^{-n^2t^2} g(w+t) \ dt. \]
Here the integral is in the sense of Remark~\ref{rem:five}, or alternatively, as $g$ is
norm continuous on any horizontal line, we can use a Riemann integral.
It follows easily that $g_n(t)\rightarrow g(t)$, uniformly in $t\in\mathbb R$,
as $n\rightarrow\infty$; similarly $g_n(t+z)\rightarrow g(t+z)$ uniformly in $t$.

We claim that
\[ g_n(w) = \frac{n}{\sqrt\pi}\int_{\mathbb R} e^{-n^2(t-w)^2} g(t) \ dt. \]
We prove this by, for each $\mu\in E^*$, considering the scalar-valued function
$w\mapsto \ip{\mu}{g_n(w)}$, and using contour deformation, and continuity.  

We now observe that $w\mapsto \frac{n}{\sqrt\pi}\int_{\mathbb R} e^{-n^2(t-w)^2} g(t)
\ dt$ is entire.  In particular, $g_n$ is norm continuous on $S(z)$.  As
$g_n\rightarrow g$ uniformly on $\mathbb R$ and $\mathbb R+z$, the Three-Lines Theorem
implies uniform convergence on all of $S(z)$.  We conclude that $g$ is norm-regular,
which implies also that $f$ is norm-regular.
\end{proof}

\begin{corollary}
Let $(\alpha_t)$ be norm-continuous on $E$.  If we use norm-regular extensions, or
weakly-regular extensions, then we arrive at the same operator $\alpha_z$.
\end{corollary}

Thus the approaches of \cite{kus1} and \cite{cz} do give the same operators.

\begin{example}
If we weaken the hypotheses of Theorem~\ref{thm:one} to only require that $t\mapsto f(t)$
be continuous, then $f$ need not be norm-regular, as the following example shows.
Set $E=c_0 = c_0(\mathbb N)$, and define $F:\overline{\mathbb D}\rightarrow E$ by
\[ F(z) = \big( F_n(z) \big)_{n\in\mathbb N}
= \big( \exp(k_n(e^{-i\pi/n}z-1)) \big)_{n\in\mathbb N}. \]
Here $(k_n)$ is a rapidly increasing sequence of integers.  Notice that
$|F_n(z)| = \exp(k_n (\re(e^{-i\pi/n}z) - 1)) \leq 1$.  Then:
\begin{itemize}
\item for $z\in\mathbb D$ we have that $e^{-i\pi/n}z\in\mathbb D$ and so
$\re(e^{-i\pi/n}z) - 1 < 0$ and hence $F_n(z)\rightarrow 0$ as $n\rightarrow\infty$;
\item If $z=e^{it}$ for $t\not\in 2\pi\mathbb Z$, then $\re(e^{-i\pi/n}z)-1
= \cos(t-\pi/n)-1 \rightarrow \cos(t)-1 < 0$ and so $F_n(z)\rightarrow 0$;
\item $|F_n(1)| = \exp(k_n(\cos(\pi/n)-1)) \rightarrow 0$ so long as $(k_n)$
increases fast enough.
\end{itemize}
Thus $(F_n(z))\in c_0$ for all $z\in\overline{\mathbb D}$.  Notice that each $F_n$
is continuous, and analytic on $\mathbb D$.

We now use that $c_0^*=\ell^1$, and for any $a=(a_n)\in\ell^1$ we have that
\[ \ip{a}{F(z)} = \sum_{n=1}^\infty a_n F_n(z) \]
converges uniformly for $z\in\overline{\mathbb D}$.  We conclude that $F$ is weakly-regular,
that is, analytic on $\mathbb D$ and weakly-continuous on $\overline{\mathbb D}$.
However,
\[ \|F(e^{i\pi/n}) - F(1)\| \geq | F_n(e^{i\pi/n}) - F_n(1)|
= | 1 - \exp(k_n(e^{-i\pi/n}-1)) |. \]
This will be large if $(k_n)$ increases rapidly.  Thus $F$ is not norm-continuous.

Finally, we can use a Mobius transformation to obtain an example defined on
the strip $S(i)$.  Indeed, $z \mapsto w = i(1-z)/(1+z)$ maps $\mathbb D$ to the
upper half-plane, and maps $\mathbb T$ to $\mathbb R \cup \{\infty\}$, and sends
$1\in\mathbb T$ to $0\in\mathbb R$.  We hence obtain $G:S(i)\rightarrow c_0$ which
is weakly-regular, with $t\mapsto G(t+i)$ norm-continuous, but $t\mapsto G(t)$ not
norm-continuous.
\end{example}

\subsection{Analytic generators}\label{sec:ag}

We call the closed operator $\alpha_{-i}$ the \emph{analytic generator} of $(\alpha_t)$.
Note that the use of $-i$ is really convention, as we can always rescale and consider
$(\alpha_{tr})$ for any non-zero $r\in\mathbb R$.  In particular, $\alpha_{-i/2}$ often
appears in applications.

We have that $\alpha_{-i}$ is a closed, densely defined operator.  The operator $\alpha_{-i}$
does determine $(\alpha_t)$, see for example Section~\ref{sec:tp} below,
and indeed one can reconstruct $(\alpha_t)$ from $\alpha_{-i}$,
see \cite[Section~4]{cz}.

\begin{example}\label{ex:two}
Let us compute the analytic extensions of the group(s) from Example~\ref{ex:one}.
If $x=(x_n)\in D(\alpha_z)\subseteq c_0(\mathbb Z)$ then for each $n$, the map
$t\mapsto e^{int} x_n$ has an analytic extension to $S(z)$, which by uniqueness
must be the map $w\mapsto e^{inw} x_n$.  Thus $\alpha_z(x) = (e^{inz}x_n)
\in c_0(\mathbb Z)$.  Reversing this, if $(e^{inz}x_n) \in c_0(\mathbb Z)$, then by
the three-lines theorem, $(x_n)\in D(\alpha_z)$.  In particular, we see that
$x=(x_n)\in D(\alpha_{-i}) \subseteq c_0(\mathbb Z)$ if and only if $(x_n)$ is in
$c_0(\mathbb Z)$ and $(x_n e^n)\in c_0(\mathbb Z)$.

Similar remarks apply to $\ell^\infty(\mathbb Z)$.  In particular, we see that
$x=(x_n)\in D(\alpha_{-i}) \subseteq \ell^\infty(\mathbb Z)$ if and only if $(x_n)$ and
$(x_n e^n)$ are bounded.

Consider $x_n = 0$ for $n<0$ and $x_n = e^{-n}$ for $n\geq 0$.  Then $x=(x_n)\in
c_0(\mathbb Z)$ but while $(x_ne^n)$ is bounded, it is not in $c_0(\mathbb Z)$.
It follows that $x\not\in D(\alpha_{-i})$ for the group acting on $c_0(\mathbb Z)$,
but $x$ is in $D(\alpha_{-i})$ for the group acting on $\ell^\infty(\mathbb Z)$.
\end{example}

\begin{example}\label{ex:three}
If we consider a one-parameter isometry group on a Hilbert space $H$, then we have the
familiar notion of a (strongly continuous) unitary group $(u_t)_{t\in\mathbb R}$.
Stone's Theorem tells us that there is a self-adjoint (possibly unbounded) operator
$A$ on $H$ with $u_t = e^{itA}$ for each $t\in\mathbb R$.  Alternatively, we can consider
the analytic generator $u_{-i}$.  \cite[Theorem~6.1]{cz} shows that $u_{-i}$, as a
(possibly unbounded) operator on $H$ is positive and injective, and equal to $e^A$.
Thus, informally, we can think of the analytic generator as the exponential of the
infinitesimal generator.
\end{example}

We now consider the case when $E=A$ is a Banach algebra, or a $C^*$-algebra.

\begin{proposition}\label{prop:one}
Let $(\alpha_t)$ be an automorphism group of a Banach algebra $A$.  Then $D(\alpha_z)$
is a subalgebra of $A$ and $\alpha_z$ a homomorphism.
\end{proposition}
\begin{proof}
Let $a,b\in D(\alpha_z)$.  We can pointwise multiply the analytic extensions
$w\mapsto \alpha_w(a)$ and $w\mapsto \alpha_w(b)$.  This is continuous, and analytic on
the interior of $S(z)$; here we use the joint norm continuity of the product on $A$.
Thus $ab\in D(\alpha_z)$ with $\alpha_z(ab)=\alpha_z(a)\alpha_z(b)$.
\end{proof}

\begin{proposition}\label{prop:star}
Let $(\alpha_t)$ be an automorphism group of a $C^*$-algebra $A$.  For $a\in D(\alpha_z)$
we have that $a^*\in D(\alpha_{\overline{z}})$ and $\alpha_{\overline{z}}(a^*) =
\alpha_z(a)^*$.
\end{proposition}
\begin{proof}
Let $f:S(z)\rightarrow A$ be the analytic extension of the orbit map for $a$.
Then $g:S(\overline{z})\rightarrow A; w\mapsto f(\overline{w})^*$ is regular (the
complex conjugate and the involution ``cancel'' to show that $g$ is analytic on the
interior of $S(\overline z)$), from which the result follows.
\end{proof}

These results become more transparent if we consider the \emph{graph} of $\alpha_z$,
\[ \mc G(\alpha_z) = \big\{ (a,\alpha_z(a)) : a\in D(\alpha_z) \big\}, \]
which is a closed subspace of $A\oplus A$, as $\alpha_z$ is closed.  Thus
$\mc G(\alpha_z)$ is a subalgebra of $A\oplus A$, and in the $C^*$-algebra case,
$\mc G(\alpha_{-i})$ has the (non-standard) involution
\[ \mc G(\alpha_{-i}) \ni (a,b) \mapsto (b^*,a^*) \in \mc G(\alpha_{-i}). \]
Here we used that $\alpha_i = \alpha_{-i}^{-1}$.

A Banach algebra $A$ which is the dual of a Banach space $A_*$ in such a way that the
product on $A$ becomes separately weak$^*$-continuous is a \emph{dual Banach algebra},
\cite{run}.
The following result is shown in \cite{z} using the idea of a \emph{spectral subspace}
from \cite{arv1, arv2, f}.  This allows us to find weak$^*$-dense subspaces (in fact,
subalgebras) on which $(\alpha_t)$ is norm continuous.  We shall later give a different,
easier proof, see Section~\ref{sec:apps}.

\begin{theorem}[{\cite[Theorem~1.6]{z}}]\label{thm:nine}
Let $A$ be a dual Banach algebra and let $(\alpha_t)$ be a weak$^*$-continuous
automorphism group of $A$.  Then $D(\alpha_z)$ is a subalgebra of $A$, and $\alpha_z$
is a homomorphism.
\end{theorem}

For a dual Banach algebra, we cannot simply copy the proof of Proposition~\ref{prop:one},
as in the weak$^*$-topology, the product is only separately continuous.  In particular,
this remark applies to von Neumann algebras.  The approach taken in \cite{kus2},
and implicitly in \cite{kus3} for example, is to use the $\sigma$-strong$^*$-topology;
\cite[Section~2.5]{fv} does the same, but with $M\subseteq\mc B(H)$ a concretely
represented von Neumann algebra, and the use of the strong topology.
Such approaches would allow the proof of Proposition~\ref{prop:one}
to now work.
Unfortunately, it is not clear if using the $\sigma$-strong$^*$-topology instead of the
weak$^*$- (that is, $\sigma$-weak-) topology gives the same set $D(\alpha_z)$.  Indeed,
is the resulting $\alpha_z$ even closed?  This issue is not addressed in \cite{kus2}.
We now show that, actually, we do obtain the same $D(\alpha_z)$.

Let $M$ be a von Neumann algebra with predual $M_*$.  For $\omega \in M_*^+$ we consider
the seminorms
\begin{gather*}
p_\omega:M\rightarrow [0,\infty), \ x\mapsto \ip{x^*x}{\omega}^{1/2}; \\
p'_\omega:M\rightarrow [0,\infty), \ x\mapsto \ip{x^*x+xx^*}{\omega}^{1/2}.
\end{gather*}
The $\sigma$-strong topology is given by the seminorms $\{p_\omega:\omega\in M_*^+\}$,
and similarly the $\sigma$-strong$^*$ topology is given by the seminorms $\{p'_\omega\}$.

\begin{lemma}\label{lem:three}
Let $E = (E_*)^*$ be a dual Banach space, let $p$ be a seminorm on $E$ for which there
exists $k>0$ with $p(x)\leq k\|x\|$ for $x\in E$, and let $z\in\mathbb C$.
Let $f:S(z)\rightarrow E$ be bounded and weak$^*$-regular, and further suppose
that $t\mapsto f(t)$ and $t\mapsto f(z+t)$ are continuous for $p$.  Then $f$ is continuous
for $p$ on all of $S(z)$.
\end{lemma}
\begin{proof}
We seek to follow the proof of Theorem~\ref{thm:one}.  Define $g(w)=e^{-w^2}f(w)$ so
again $g$ is weak$^*$-regular and $t\mapsto g(t)$, $t\mapsto g(z+t)$ are
uniformly continuous for $p$.  For $n>0$ we can again define $g_n:S(z)\rightarrow E$ by
\[ g_n(w) = \frac{n}{\sqrt\pi} \int_{\mathbb R} e^{-n^2t^2} g(w+t) \ dt, \]
the integral converging in the weak$^*$ sense.  We see that
$g_n(t)\rightarrow g(t)$ uniformly in $t$, for the seminorm $p$, and similarly for
$g_n(t+z)\rightarrow g(t+z)$.

We again have the alternative expression $g_n(w) = \frac{n}{\sqrt\pi} \int_{\mathbb R}
\exp(-n^2(t-w)^2) f(t) \ dt$.  Thus $g_n$ extends to an analytic function on $\mathbb C$;
in particular $g_n$ is locally given by a $\|\cdot\|$-convergent power series, which is
hence also $p$-convergent.  It follows that $g_n$ is $p$-continuous on $S(z)$.
As $p(g_n-g)\rightarrow 0$ uniformly on $\mathbb R$ and $\mathbb R+z$, the Three-Lines
Theorem\footnote{See Appendix~\ref{add:2} below for clarification.} implies uniform convergence on all of $S(z)$.
Thus $g$ is $p$-continuous on $S(z)$, and the same is true of $f$.
\end{proof}

\begin{lemma}\label{lem:four}
Let $M$ be a von Neumann algebra and let $(\alpha_t)$ be a weak$^*$-continuous
automorphism group.  For each $x\in M$ the map $\mathbb R\rightarrow M; t \mapsto
\alpha_t(x)$ is $\sigma$-strong$^*$ continuous.
\end{lemma}
\begin{proof}
Let $\omega\in M_*^+$ and $t\in\mathbb R$.  Then for $x\in M$,
\begin{align*}
\lim_{t\rightarrow0}\ip{(\alpha_t(x)-x)^*(\alpha_t(x)-x)}{\omega}
&= \lim_{t\rightarrow 0}
   \ip{\alpha_t(x^*x) - x^*\alpha_t(x) - \alpha_t(x^*)x + x^*x}{\omega} \\
&= \lim_{t\rightarrow 0} \ip{\alpha_t(x^*x)+x^*x}{\omega} - \ip{\alpha_t(x)}{\omega x^*}
   - \ip{\alpha_t(x^*)}{x\omega} \\
&= 2 \ip{x^*x}{\omega} - \ip{x}{\omega x^*} - \ip{x^*}{x\omega} = 0,
\end{align*}
where we used repeatedly that $\alpha_t$ is a $*$-homomorphism, and that $M_*$ is
an $M$-module, and of course that $(\alpha_t)$ is weak$^*$-continuous.
Similarly, $\ip{(\alpha_t(x)-x)(\alpha_t(x)-x)^*}{\omega} \rightarrow 0$ as
$t\rightarrow 0$.  Thus $\alpha_t(x)\rightarrow x$ as $t\rightarrow 0$, in the
$\sigma$-strong$^*$ topology.
\end{proof}

\begin{theorem}\label{thm:10a}
Let $M$ be a von Neumann algebra, let $(\alpha_t)$ be a weak$^*$-continuous
automorphism group, let $x\in M$, and let $f:S(z)\rightarrow M$ be a weak$^*$-regular
extension of $t\mapsto\alpha_t(x)$.
Then $f$ is continuous for the $\sigma$-strong$^*$ (and so $\sigma$-strong) topology.
\end{theorem}
\begin{proof}
By Lemma~\ref{lem:four}, $t\mapsto f(t)=\alpha_t(x)$ and $t\mapsto f(z+t)
= \alpha_t(f(z))$ are $\sigma$-strong$^*$
continuous.  The result now follows from Lemma~\ref{lem:three} applied to the 
seminorms $p'_{\omega}$ for $\omega\in M_*^+$.
\end{proof}

We conclude that the definition of $\alpha_z$ from \cite{kus2} does agree with the
definition in \cite{cz}, and we are free to use either the $\sigma$-strong$^*$
topology, or the weak$^*$ topology.  If $M\subseteq\mc B(H)$ and we use the strong
topology, the same remarks apply.

\subsection{Duality}\label{sec:duality}

Let $E$ be a Banach space and let $(\alpha_t)$ be a norm-continuous one-parameter group
of isometries of $E$.  For each $t$ let $\alpha_t^* \in \mc B(E^*)$ be the Banach space
adjoint.  Then $(\alpha_t^*)$ is a weak$^*$-continuous one-parameter group of isometries
of $E^*$.

Similarly, let $E=(E_*)^*$ be a dual Banach space and let $(\alpha_t)$ be a
weak$^*$-continuous one-parameter group of isometries of $E$. For each $t$, as $\alpha_t$
is weak$^*$-continuous it has a pre-adjoint $\alpha_{*,t}$.  As
\[ \ip{\alpha_t(x)}{\mu} = \ip{x}{\alpha_{*,t}(\mu)} \qquad (x\in E, \mu\in E_*) \]
it is easy to see that $(\alpha_{*,t})$ is a one-parameter group of isometries of $E_*$
which is weakly-continuous, and hence which is norm-continuous.

We recall
that when $T:D(T)\subseteq E\rightarrow F$ is an operator between Banach spaces, then the
\emph{adjoint} of $T$ is defined by setting $\mu\in D(T^*)\subseteq F^*$ when there exists
$\lambda\in E^*$ with $\ip{\mu}{T(x)} = \ip{\lambda}{x}$ for $x\in D(T)$.  In this case, we
set $T^*(\mu) = \lambda$.  This is more easily expressed in terms of graphs.  Define
$j:E\oplus F \rightarrow F \oplus E$ by $j(x,y) = (-y,x)$.  Then $\mc G(T^*)$ is equal to
\[ (j\mc G(T))^\perp = \{ (\mu,\lambda) \in F^*\oplus E^* :
\ip{(\mu,\lambda)}{(-T(x),x)}=0 \ (x\in D(T)) \}. \]
That $\mc G(T^*)$ is the graph of an \emph{operator} is equivalent to $T$ being densely
defined; in this case, $\mc G(T^*)$ is always weak$^*$-closed.  We can reverse this
construction, starting with an operator $S:D(S)\subseteq F^*\rightarrow E^*$ and forming
$S_*:D(S_*)\subseteq E\rightarrow F$ by $\mc G(S_*) = {}^\perp (jD(S))$.  Then $S_*$ is
an operator exactly when $S$ is weak$^*$-densely defined, and $S_*$ is always closed.
Thus, if $T$ is closed and densely-defined, then $S=T^*$ is weak$^*$-closed and densely
defined, and $S_* = T$.  We are actually unaware of a canonical reference for this
construction (which clearly parallels the very well-known construction for Hilbert space
operators) but see \cite[Section~5.5, Chapter~III]{kat} for example.

The following is shown in \cite{z} using a very similar argument to the proof that the
generator, of a weak$^*$-continuous group, is weak$^*$-closed.  We give a different proof,
which relies on the closure result, and which will be presented below in
Section~\ref{sec:apps}.  In fact, given the discussion above, this theorem is
effectively equivalent to knowing that the generator is closed.  

\begin{theorem}[{\cite[Theorem~1.1]{z}}]\label{thm:two}
Let $(\alpha_t)$ on $E$ and $(\alpha_t^*)$ on $E^*$ be as above.  For any $z$, we form
$\alpha_z$ using $(\alpha_t)$, and form $\alpha_z^{E^*}$ using $(\alpha_z^*)$.
Then $\alpha_z ^* = \alpha_z^{E^*}$.
\end{theorem}

We remark that we have used this result before, e.g. \cite[Appendix]{bds}, but without
sufficient justification as to why $\alpha_z^* = \alpha^{E^*}_z$.  Similar ideas, but
without the machinery of using $(\alpha_t^*)$, are considered in \cite[Proposition~1.24,
Proposition~2.44]{kus1}.

\section{Smearing}\label{sec:smearing}

We now want to present some ideas from the Appendix of \cite{mnw}, which only considered
norm-continuous one-parameter groups.  We shall verify that
the ideas continue to work for weak$^*$-continuous one-parameter groups.  This is fairly
routine, excepting perhaps Proposition~\ref{prop:prop1}, but we feel it is worth giving
the details, as we think the techniques and results are interesting.  We also streamline
the proof of the main technical lemma, directly invoking the classical Wiener Theorem,
instead of using Distribution theory.  We remark that the use of convolution algebra ideas
goes back to at least \cite{arv1,arv2} and \cite{f}.

Let $(\alpha_t)$ be a one-parameter group of isometries on $E$; we shall consider both
the case when $(\alpha_t)$ is norm-continuous, and when $E=(E_*)^*$ is a dual space and
$(\alpha_t)$ is weak$^*$-continuous.
Given $n>0$ define $\mc R_n:E\rightarrow E$ by
\[ \mc R_n(x) = \frac{n}{\sqrt\pi} \int_{\mathbb R} \exp(-n^2t^2) \alpha_t(x) \ dt. \]
The integral converges in norm, or the weak$^*$-topology, according to context.
As in the proof of Theorem~\ref{thm:one}, a contour deformation argument shows that
for any $z\in\mathbb C$, $\mc R_n(x)\in D(\alpha_z)$ with
\[ \alpha_z(\mc R_n(x)) = \frac{n}{\sqrt\pi} \int_{\mathbb R} \exp(-n^2(t-z)^2)
\alpha_t(x) \ dt. \]
Furthermore, if already $x\in D(\alpha_z)$ then $\alpha_z(\mc R_n(x)) =
\mc R_n(\alpha_z(x))$.

This concept of smearing is very standard in arguments involving analytic generators,
but it is common to consider the limit as $n\rightarrow\infty$.  For example, for any
$x\in E$ we have that $\mc R_n(x)\rightarrow x$ as $n\rightarrow\infty$ (again, in norm or
the weak$^*$-topology) and so this shows that $D(\alpha_z)$ is dense.  In the following,
the point is to show that it is possible to work with $\mc R_n$ for a fixed $n$.

In the following, a subspace $X\subseteq E$ is \emph{$(\alpha_t)$-invariant} when
$\alpha_t(x)\in X$ for each $x\in X, t\in\mathbb R$.  The following is immediate
from the construction of $\mc R_n$ as a vector-valued integral.

\begin{lemma}\label{lem:one}
For each $x\in E$, we have that $\mc R_n(x)$ is contained in the smallest
$(\alpha_t)$-invariant, closed (norm or weak$^*$ as appropriate) subspace of $E$
containing $x$.
\end{lemma}

The following result is somewhat less expected.

\begin{lemma}\label{lem:two}
For each $x\in E$ and $n>0$, we have that $x$ is contained in the smallest
$(\alpha_t)$-invariant, closed (norm or weak$^*$ as appropriate) subspace of $E$
containing $\mc R_n(x)$.
\end{lemma}
\begin{proof}
Choose $\mu\in E^*$ or $E_*$ as appropriate with $\ip{\mu}{\alpha_t(\mc R_n(x))}=0$
for each $t\in\mathbb R$.  By Hahn-Banach, it suffices to show that $\ip{\mu}{x}=0$.

Define $f,g:\mathbb R\rightarrow\mathbb C$ by
\[ f(t)=\ip{\mu}{\alpha_t(x)}, \quad g(t)=\ip{\mu}{\alpha_t(\mc R_n(x))}
\qquad (t\in\mathbb R). \]
Then $f$ and $g$ are bounded continuous functions, and
\begin{align*}
g(t) &= \frac{n}{\sqrt\pi} \int_{\mathbb R} \exp(-n^2s^2) \ip{\mu}{\alpha_{t+s}(x)}
\ ds
= \frac{n}{\sqrt\pi} \int_{\mathbb R} \exp(-n^2(s-t)^2) \ip{\mu}{\alpha_{s}(x)}
\ ds \\
&= \frac{n}{\sqrt\pi} \int_{\mathbb R} \exp(-n^2(s-t)^2) f(s) \ ds.
\end{align*}
Thus $g$ is the convolution of $\varphi$ with $f$, where $\varphi(s)=\frac{n}{\sqrt\pi}
\exp(-n^2s^2)$, so that $\varphi\in L^1(\mathbb R)$.

So, we wish to show that if $\varphi * f=0$ then $f=0$.  Given $F\in L^\infty(\mathbb R)$
and $a,b\in L^1(\mathbb R)$, a simple calculation shows that
\[ \ip{F\cdot a}{b} = \ip{F}{a*b} = \ip{F*\check a}{b}, \]
where here $F\cdot a$ is the usual dual module action of $L^1(\mathbb R)$ on
$L^\infty(\mathbb R) = L^1(\mathbb R)^*$, and $\check a\in L^1(\mathbb R)$ is the
function defined by $\check a (t) = a(-t)$.  As $f\in C^b(\mathbb R)\subseteq
L^\infty(\mathbb R)$, by Hahn-Banach, we see that $\varphi * f=0$ is equivalent to
$\ip{f}{\check \varphi * g} = 0$ for each $g\in L^1(\mathbb R)$.  To conclude that $f=0$
it hence suffices to show that $\{\check \varphi * g : g\in L^1(\mathbb R) \}$ is dense in
$L^1(\mathbb R)$.  This is equivalent to showing that the translates of $\check\varphi$
are linearly dense in $L^1(\mathbb R)$.  In turn, this follows immediately from Wiener's
Theorem (see \cite[Theorem~II]{w} or \cite[Theorem~9.4]{rudin}) as $\check\varphi =
\varphi$ has a nowhere vanishing Fourier transform.  We remark that a different approach
to this result would be to use Eymard's Fourier algebra \cite{eym} (where a related result
about the action of $A(G)$ on $VN(G)$ holds for all locally compact groups $G$) but as we
need simply the most classical version, we shall not give further details.
\end{proof}

In the following, $n>0$ is any (fixed) number.

\begin{proposition}\label{prop:prop3}
Let $D\subseteq E$ be an $(\alpha_t)$-invariant subspace.  Then $\mc{R}_n(D)
=\{\mc R_n(x):x\in D\}$ and
$D$ have the same (norm, or weak$^*$) closure.
\end{proposition}
\begin{proof}
As $\alpha_t$ commutes with $\mc{R}_n$ for each $t$, it follows that $\mc{R}_n(D)$
is $(\alpha_t)$-invariant.  For each $x\in D$, the closure of $\mc{R}_n(D)$ contains
the smallest closed $(\alpha_t)$-invariant subspace containing $\mc{R}_n(x)$, so by
Lemma~\ref{lem:one}, $x\in \overline{\mc{R}_n(D)}$, and hence $\overline{D}
\subseteq \overline{\mc{R}_n(D)}$.  The reverse inclusion follows similarly from
Lemma~\ref{lem:two}.
\end{proof}

The following gives a criteria for being a member of the graph of $\alpha_z$.

\begin{proposition}\label{prop:prop2}
Let $x,y\in E$ and $z\in\mathbb C$ with $\alpha_z(\mc{R}_n(x)) = \mc{R}_n(y)$.
Then $x\in D(\alpha_z)$ with $\alpha_z(x) = y$.
\end{proposition}
\begin{proof}
Consider the graph $\mc G(\alpha_z) = \{ (x,\alpha_z(x)):
x\in D(\alpha_z) \}$, a closed subspace of $E\oplus E$.  The one-parameter group
$\beta_t = \alpha_t \oplus \alpha_t$ on $E\oplus E$ leaves $\mc G(\alpha_z)$ invariant.
The hypothesis is that $(\mc R_n(x), \mc R_n(y)) \in \mc G(\alpha_z)$, and a simple
calculation shows that the ``smearing operator'' for $\beta$ is $\mc R_n \oplus \mc R_n$.
Thus Lemma~\ref{lem:one} applied to $(\beta_t)$ shows that $(x,y) \in \mc G(\alpha_z)$,
as required.
\end{proof}

In the norm-continuous case, we equip $D(\alpha_z)$ with the graph norm,
$\|x\|_{\mc G} = \|x\| + \|\alpha_z(x)\|$ (which is the $\ell^1$ norm; but clearly any
complete norm would work).  In the weak$^*$-continuous case, equip $D(\alpha_z)$
with the restriction of the weak$^*$-topology on $E\oplus_1 E = (E_*\oplus_\infty E_*)^*$
(again here any suitable norm on $E_*\oplus E_*$ would suffice).  In either case, we
speak of the \emph{graph topology} on $D(\alpha_z)$.

\begin{proposition}\label{prop:prop1}
Let $D_1\subseteq D_2\subseteq E$ be subspaces with $D_1$ dense in $D_2$,
and let $z\in\mathbb C$. Then $\mc{R}_n(D_1)\subseteq \mc{R}_n(D_2)$ is dense in the graph
topology (or, equivalently, the closure of $\alpha_z$ restricted to $\mc{R}_n(D_1)$ agrees
with the closure of $\alpha_z$ restricted to $\mc{R}_n(D_2)$).
\end{proposition}
\begin{proof}
We show the weak$^*$-continuous case, the norm-continuous case being easier
(and already shown in \cite{mnw}).
Let $(\alpha_{*,t})$ be the one-parameter group on $E_*$ given by $(\alpha_t)$,
see the discussion in Section~\ref{sec:duality}.

For $x\in D_2$ we seek a net $(y_i)\subseteq D_1$ with
$\mc R_n(y_i)\rightarrow\mc R_n(x)$ weak$^*$, and with
$\alpha_z(\mc R_n(y_i))\rightarrow\alpha_z(\mc R_n(x))$ weak$^*$.

Let $M\subseteq E_*$ be a finite set, and $\epsilon>0$.  We seek $y\in D_1$
with
\[ \Big| \int_{\mathbb R} e^{-n^2t^2} \ip{\alpha_t(x-y)}{\mu} \ dt \Big| <
\epsilon \qquad (\mu\in M), \]
and with
\[ \Big| \int_{\mathbb R} e^{-n^2(t-z)^2} \ip{\alpha_t(x-y)}{\mu} \ dt \Big| <
\epsilon \qquad (\mu\in M). \]
These inequalities would follow if we can show that $|\ip{\alpha_t(x-y)}{\mu}|
< \epsilon'$ for $|t|\leq K, \mu\in M$, where $K,\epsilon'$ depend only on
$\epsilon$ (and on $z$ which is fixed).  This is equivalent to
\[ |\ip{x-y}{\alpha_{*,t}(\mu)}| < \epsilon' \qquad(\mu\in M, |t|\leq K). \]
Now, the set $\{ \alpha_{*,t}(\mu) : |t|\leq K, \mu\in M \}$ is compact in $E_*$,
because $M$ is finite and $t\mapsto\alpha_{*,t}(\mu)$ is norm continuous.  Thus
$D_1$ being weak$^*$-dense in $D_2$ is enough to ensure we can choose such a $y$ as
required.
\end{proof}

\begin{theorem}\label{thm:thm1}
Let $D\subseteq E$ be an $(\alpha_t)$-invariant subspace, let $z\in\mathbb C$,
and suppose that $D\subseteq D(\alpha_z)$.  If $D$ is dense in $E$, then
$D$ is a core for $\alpha_z$.
\end{theorem}
\begin{proof}
As in the proof of Proposition~\ref{prop:prop2} we shall again consider $(\beta_t)$ acting
on $\mc G(\alpha_z)$.  As $D$ is $(\alpha_t)$-invariant, it follows that
$\mc G(\alpha_z|_D) = \{ (x,\alpha_z(x)) : x\in D \}$ is $(\beta_t)$-invariant.  Set
$D' = \{ (x,\alpha_z(x)) : x\in D\}$.  Applying Proposition~\ref{prop:prop3} to $(\beta_t)$, it
follows that the closures of $D'$ and $\mc{R}_n(D')$ agree.  Equivalently, the closure of
$\alpha_z|_D$ agrees with the closure of $\alpha_z|_{\mc{R}_n(D)}$.
Apply this with $D=D(\alpha_z)$ itself to see that
\[ \overline{\alpha_z|_{\mc{R}_n(D(\alpha_z))}} = \alpha_z. \]
As $\mc{R}_n(D(\alpha_z)) \subseteq \mc{R}_n(E) \subseteq D(\alpha_z)$, it follows that
\[ \overline{\alpha_z|_{\mc{R}_n(E)}} = \alpha_z. \]
As $D$ is dense in $E$, it now follows from Proposition~\ref{prop:prop1} that
$\mc{R}_n(D)$ is a core for $\alpha_z$, because $\mc{R}_n(E)$ is a core.  Then finally
applying the first part of the proof again shows that $D$ itself is a core for
$\alpha_z$, as required.
\end{proof}

We end this section with a result purely about weak$^*$-continuous one-parameter groups.

\begin{proposition}\label{prop:5}
Let $(\alpha_t)$ be a weak$^*$-continuous group on a dual space $E=(E_*)^*$.  For
any $n$ and $x\in E$, the map $\mathbb R\rightarrow E; t\mapsto \alpha_t(\mc{R}_n(x))$
is norm continuous.
\end{proposition}
\begin{proof}
For any fixed $n$, notice that the Gaussian kernel $\varphi(t) =
\frac{n}{\sqrt\pi} \exp(-n^2t^2)$ is in $L^1(\mathbb R)$.
As the translation action of $\mathbb R$ on $L^1(\mathbb R)$ is strongly continuous,
we see that
\[ \lim_{s\rightarrow 0} \frac{n}{\sqrt\pi} \int_{\mathbb R}
|\exp(-n^2t^2) - \exp(-n^2(t-s)^2)| \ dt = 0. \]
It then follows that
\[ \|\mc{R}_n(x) - \alpha_s(\mc{R}_n(x))\|
\leq \frac{n}{\sqrt\pi} \int_{\mathbb R}
|\exp(-n^2t^2) - \exp(-n^2(t-s)^2)| \|\alpha_t(x)\| \ dt, \]
which converges to $0$ as $s\rightarrow 0$, uniformly in $\|x\|$.
\end{proof}

\section{Applications}\label{sec:apps}

The previous section drew some conclusions about the operators $\mc R_n$.
We now wish to present a number of applications of these conclusions, which we think
demonstrates the power of these ideas.  We start by giving the proof that
``the dual of the generator is the generator of the dual group''.

\begin{proof}[Proof of Theorem~\ref{thm:two}]
We fix $n>0$, and then make the key, but easy, observation that the Banach space adjoint
$\mc R_n^*$ of $\mc R_n$ is the ``smearing operator'' of the dual group $(\alpha_t^*)$.
By Theorem~\ref{thm:thm1}, we know that $\mc R_n(E)$ is a core for $\alpha_z$, that is,
$\{ (\mc R_n(x), \alpha_z\mc R_n(x)) : x\in E \}$ is (norm) dense in $\mc G(\alpha_z)$.
Similarly, using the key observation,
$\{ (\mc R_n^*(\mu), \alpha_z^{E^*}\mc R_n^*(\mu)) : \mu\in E^* \}$ is
weak$^*$-dense in $\mc G(\alpha^{E^*}_z)$.  Notice further that if we define
$\mc T_n = \alpha_z\mc R_n$,
then $\mc T_n(x) = \frac{n}{\sqrt\pi}\int_{\mathbb R} \exp(-n^2(t-z)^2) \alpha_t(x)
\ dt$, from which it follows that $\mc T_n^* = \alpha^{E^*}_z \mc R_n^*$.

Let $(\mu,\lambda) \in \mc G(\alpha_z^*)$.  This is equivalent to $(-\lambda,\mu)
\in \mc G(\alpha_z)^\perp$, which by the previous paragraph is equivalent to
\[ 0 = \ip{-\lambda}{\mc R_n(x)} + \ip{\mu}{\alpha_z\mc R_n(x)}
= \ip{-\mc R_n^*(\lambda) + \mc T_n^*(\mu)}{x}
\qquad (x\in E). \]
That is, equivalent to $\mc T_n^*(\mu) = \mc R_n^*(\lambda)$.
By Proposition~\ref{prop:prop2}, this is equivalent to $(\mu,\lambda)
\in \mc G(\alpha_z^{E^*})$, as required.
\end{proof}

We now consider Theorem~\ref{thm:nine} which shows that if $(A,A_*)$ is a dual Banach
algebra and $(\alpha_t)$ a weak$^*$-continuous automorphism group of $A$, then
$\mc G(\alpha_z)$ is a subalgebra of $A\oplus A$.

\begin{lemma}\label{lem:five}
Let $(A,A_*)$ be a dual Banach algebra and let $X\subseteq A$ be a (possibly not closed)
subalgebra.  Then the weak$^*$-closure of $X$ is a subalgebra.
\end{lemma}
\begin{proof}
Let $\overline{X}$ be the weak$^*$-closure of $X$.  Then $\overline{X}$ is the dual of
$A_* / {}^\perp X$, and $\overline{X} = ({}^\perp X)^\perp$.
That $A$ is a dual Banach algebras is equivalent to $A_*$ being an $A$-bimodule for the
natural actions coming from the product on $A$.

For $\mu\in {}^\perp X$ and $a,b\in X$, we have that
$\ip{b}{\mu\cdot a} = \ip{ab}{\mu} = 0$ as $X$ is a subalgebra.
Thus $\mu\cdot a\in {}^\perp X$ for each $a\in X$, and 
similarly, $X\cdot {}^\perp X \subseteq {}^\perp X$.
Now let $x\in\overline{X}$, so for $a\in X$, we have that
$\ip{a}{x\cdot\mu} = \ip{ax}{\mu} = \ip{x}{\mu\cdot a} = 0$,
as $x\in({}^\perp X)^\perp$.  Thus $x\cdot\mu \in {}^\perp X$, and similarly
$\mu\cdot x \in {}^\perp X$.
Finally, for $x,y\in\overline{X}$ and $\mu\in {}^\perp X$, we have that
$\ip{xy}{\mu} = \ip{y}{\mu\cdot x} = 0$.  Thus $xy\in \overline{X}$ as required.
\end{proof}

\begin{proof}[Proof of Theorem~\ref{thm:nine}]
Fix $n>0$ and let $\mc R$ be the smearing operator $\mc R_n$
defined on $A$ using $(\alpha_t)$.
For $a\in A$ we have that $\mc R(a)$ is analytic so in particular
$w\mapsto \alpha_w(\mc R(a))$ is norm continuous.  As in the proof of
Proposition~\ref{prop:one} it follows that for $a,b\in A$ we have that
$w\mapsto \alpha_w(\mc R(a)) \alpha_w(\mc R(b))$ is analytic and extends $t\mapsto
\alpha_t(\mc R(a))\alpha_t(\mc R(b)) = \alpha_t(\mc R(a)\mc R(b))$.  It follows that
$\mc R(a)\mc R(b) \in D(\alpha_z)$ with $\alpha_z(\mc R(a)\mc R(b)) = \alpha_z(\mc R(a))
\alpha_z(\mc R(b))$.

By Theorem~\ref{thm:thm1} we know that $X = \{ (\mc R(a), \alpha_z(\mc R(a))) :
a\in A \}$ is weak$^*$-dense in $\mc G(\alpha_z)$.  We have just proved that $X$ is
a subalgebra of $A\oplus A$.  If we consider, say, $A\oplus_\infty A$, then this is a
dual Banach algebra with predual $A_*\oplus_1 A_*$.
The result follows from Lemma~\ref{lem:five}.
\end{proof}

A recurring theme in much of the rest of the paper is the following setup.
Let $A$ be a C$^*$-algebra which
is weak$^*$-dense in a von Neumann algebra $M$.  Suppose that $(\alpha_t)$ is a
one-parameter automorphism group of $M$ which restricts to a (norm-continuous)
one-parameter automorphism group of $A$.  To avoid confusion, we shall write
$(\alpha_t^M)$ and $(\alpha_t^A)$, and similarly for the analytic extensions.

\begin{proposition}\label{prop:three}
Let $M,A$ and $(\alpha_t)$ be as above, and let $z\in\mathbb C$.
Then $D(\alpha_z^A)$ is a core for $D(\alpha_z^M)$.
\end{proposition}
\begin{proof}
As $D(\alpha_z^A)$ is norm dense in $A$, it is also weak$^*$-dense in $M$.  The result
now follows immediately from Theorem~\ref{thm:thm1}, as clearly $\mc D(\alpha_z^A)$ is
$(\alpha_t^M)$-invariant, because it is $(\alpha_t^A)$-invariant.
\end{proof}

\begin{proposition}\label{prop:4}
Let $M,A$ and $(\alpha_t)$ be as above, and let $z\in\mathbb C$.
Let $a\in A$ be such that $a\in D(\alpha_z^M)$.  Then $a\in D(\alpha_z^A)$ if
and only if $\alpha_z^M(a)\in A$.

In other words, if $\mc G^M \subseteq M\oplus M$ is the graph of $\alpha_z^M$, and
$\mc G^A \subseteq A\oplus A$ is the graph of $\alpha_z^A$, then
$\mc G^M \cap (A\oplus A) = \mc G^A$.
\end{proposition}
\begin{proof}
By the definition of analytic continuation, it follows that $\mc G^A \subseteq
\mc G^M$ for the inclusion $A\oplus A\subseteq M\oplus M$.  Thus, if $a\in D(\alpha^A_z)$
then $\alpha_z^M(a) = \alpha_z^A(a) \in A$.

Conversely, suppose that $a\in D(\alpha_z^M)$ with $b = \alpha_z^M(a) \in A$.
As $(\alpha_t)$ is norm continuous on $A$, we have
that both $\mc{R}_n(a), \mc{R}_n(b) \in A$, and we obtain the same elements if we
consider a norm converging integral, or a weak$^*$-converging integral.  In $M$,
we have that $\alpha_z^M(\mc{R}_n(a)) = \mc{R}_n(\alpha_z^M(a)) = \mc{R}_n(b)$.
However, $\alpha_z^M(\mc{R}_n(a))$ is equal to another integral which we can
consider converging in $A$.  Thus Proposition~\ref{prop:prop2} applied to $A$ gives
the result.
\end{proof}

A more abstract result about ``inclusions'' of general one-parameter groups could be
formulated and proved in a similar way; compare also Proposition~\ref{prop:foura} below.
We remark that ``quotients'' of one-parameter groups seems a more subtle issue, see
Section~\ref{sec:quot} below.

\begin{example}\label{ex:four}
Consider Examples~\ref{ex:one} and~\ref{ex:two}.  There we considered a one-parameter
isometric group acting on the $C^*$-algebra $c_0(\mathbb Z)$ and the von Neumann algebra
$\ell^\infty(\mathbb Z)$.  Of course, these groups were not automorphism groups.

Consider the Hilbert space $H=\ell^2(\mathbb Z)$ with orthonormal basis
$(\delta_n)_{n\in\mathbb Z}$.  Let $(p_n)$ be a sequence of non-zero positive numbers,
and define the (in general unbounded) positive non-degenerate operator $P$ on $H$ by
$P(\delta_n) = p_n \delta_n$.  Then $P^{it}(\delta_n) = p_n^{it}\delta_n$
for $t\in\mathbb R$.

Now consider $\mc B(H\oplus H)$, the bounded operators on $H\oplus H$, which we identify
with $2\times 2$ matrices with entries in $\mc B(H)$.  Let $u_t = \begin{pmatrix}
P^{it} & 0 \\ 0 & 1 \end{pmatrix}$ a unitary on $H\oplus H$ with $u_t^* = u_{-t}$.
Then $x\mapsto \tau_t(x) = u_t x u_{-t}$ defines a weak$^*$-continuous automorphism group
on $\mc B(H\oplus H)$.  We have that
\[ u_t \begin{pmatrix} a & b \\ c & d \end{pmatrix} u_{-t}
= \begin{pmatrix} P^{it} & 0 \\ 0 & 1 \end{pmatrix}
\begin{pmatrix} a & b \\ c & d \end{pmatrix}
\begin{pmatrix} P^{-it} & 0 \\ 0 & 1 \end{pmatrix}
= \begin{pmatrix} P^{it}aP^{-it} & P^{it}b \\ cP^{-it} & d \end{pmatrix}. \]

Now, $c_0(\mathbb Z)$ acts on $\ell^2(\mathbb Z)$ by multiplication, and commutes with
$P$, so
\[ u_t \begin{pmatrix} a & b \\ c & d \end{pmatrix} u_{-t}
= \begin{pmatrix} a & \alpha_t(b) \\ \alpha_{-t}(c) & d \end{pmatrix}
\qquad (a,b,c,d\in c_0(\mathbb Z)), \]
where $\alpha_t(a) = (p_n^{it} a_n)$ for $t\in\mathbb R, a=(a_n)\in c_0(\mathbb Z)$.
Thus $\alpha_t$ is a generalisation of the group considered in
Examples~\ref{ex:one} and~\ref{ex:two}.  So $(\tau_t)$ restricts to a (norm-continuous)
automorphism group of $\mathbb M_2(c_0(\mathbb Z))$.
We can clearly replace $c_0(\mathbb Z)$ by $\ell^\infty(\mathbb Z)$ if we also replace
the norm topology by the weak$^*$ topology.

We have hence embedded the one-parameter \emph{isometry} group $(\alpha_t)$ into the
one-parameter \emph{automorphism} group $(\tau_t)$.  In particular, Example~\ref{ex:two}
shows that Proposition~\ref{prop:4} is false if we drop the condition that
$\alpha^M_z(a)\in A$ (that is, $A \cap D(\alpha^M_z)$ can be strictly larger than
$D(\alpha^A_z)$).

The reader should compare this counter-example with Theorem~\ref{thm:seven} below.
\end{example}

Let $(u_t)$ be a strongly continuous unitary group on a Hilbert space $H$, and define
$\tau_t(x) = u_t x u_{-t}$ for $x\in\mc B(H)$, so that $(\tau_t)$ is a weak$^*$-continuous
automorphism group.  Such groups were studied in \cite[Section~6]{cz}.

\begin{theorem}[{\cite[Theorem~6.2]{cz}}]\label{thm:three}
With $\tau_t(x) = u_t x u_{-t}$ acting on $\mc B(H)$, we have that $x\in D(\tau_z)
\subseteq\mc B(H)$ if and only if $D(u_z x u_{-z})$ is a core for $u_{-z}$ and
$u_z x u_{-z}$ is bounded.  If $x\in D(\tau_z)$ then $D(u_z x u_{-z}) = D(u_{-z})$
and $\tau_z(x)$ is the closure of $u_z x u_{-z}$.
\end{theorem}

We recall that $D(u_z x u_{-z}) = \{\xi\in D(u_{-z}) : xu_{-z}\xi \in D(u_z) \}$.
If $M\subseteq\mc B(H)$ is a von Neumann algebra, and $\tau_t(M)\subseteq M$ for each
$t\in\mathbb R$, then we obtain the restricted automorphism group $(\tau^M_t)$.
If we are given an automorphism group $(\alpha_t)$ on $M$, and $(u_t)$ on $H$, then
a criteria for when $(\alpha_t)$ arises as the restriction of $(\tau_t)$, given
in terms of $u_{-i}$ and $\alpha_{-i}$, is \cite[Corollary~2.5]{z}.
Alternatively, for a criteria for when $\tau_t(M)\subseteq M$, given in terms of $M$
and $u_{-i}$, see \cite[Theorem~3.5]{z1}, which follows \cite{w1,w2}.

Let us record that the above characterisation also applies to $\mc D(\tau^M_z)$;
notice that the conclusion is stronger than Proposition~\ref{prop:4}.

\begin{proposition}\label{prop:foura}
Consider $(\tau^M_t)$ as above.  Then $x\in D(\tau^M_z)$ if and only if $x\in M$ with
$D(u_z x u_{-z})$ is a core for $u_{-z}$ and $u_z x u_{-z}$ is bounded.
If $x\in D(\tau_z^M)$ then $D(u_z x u_{-z}) = D(u_{-z})$ and $\tau_z^M(x)$ is
the closure of $u_z x u_{-z}$.
\end{proposition}
\begin{proof}
This should be compared with \cite[Corollary~2.5]{z} mentioned above.
Given such an $x$, let $y$ be the closure of $u_z x u_{-z}$.  By the previous theorem,
there is a weak$^*$-regular map $f:S(z) \rightarrow B(H)$ with $f(t) = \tau^M_t(x)$
for $t\in\mathbb R$, and with $f(z) = y$.  For any $\omega \in {}^\perp M \subseteq
\mc B(H)_*$ we have that $S(z)\rightarrow\mathbb C; w\mapsto \ip{f(w)}{\omega}$ is
regular, and identically $0$ on $\mathbb R$, and so vanishes everywhere.
Thus $f$ maps $S(z)$ into $({}^\perp M)^\perp = M$ and so $y \in M$,
so $(x,y) \in \mc G(\tau^M_z)$ as required.
\end{proof}

\subsection{Tomita-Takesaki theory}

We now make some remarks about Tomita-Takesaki theory.  Let $M$ be a von Neumann
algebra with $\varphi$ a normal semi-finite faithful weight on $M$, see
\cite[Chapter~VII]{tak2}.  Let $\mf n_\varphi = \{ x\in M : \varphi(x^*x)<\infty \}$
and let $\Lambda:\mf n_\varphi\rightarrow H$ be the GNS map.  Then $\mf A =
\Lambda(\mf n_\varphi \cap \mf n_\varphi^*)$ is a full left Hilbert algebra, and
Tomita-Takesaki theory gives rise to the modular conjugation $J$ on $H$, and the modular
operator $\Delta$ which implements the modular automorphism group
$\sigma_t(\cdot) = \Delta^{it}(\cdot) \Delta^{-it}$.

There is a direct link between $\sigma_{-i}$ and $\varphi$, which we quote for the sake
of interest.

\begin{proposition}[{See \cite[Section~3]{h1} or \cite[Theorem~3.25, Chapter~VIII]{tak2}}]
\label{prop:9}
For $a,b\in M$ the following are equivalent:
\begin{enumerate}
\item\label{prop:9:1} $(a,b)\in\mc G(\sigma_{-i})$;
\item\label{prop:9:2} $a\mf n_\varphi^* \subseteq \mf n_\varphi^*,
\mf n_\varphi b \subseteq \mf n_\varphi$ and $\varphi(ax) = \varphi(xb)$
for $x\in \mf n_\varphi^*\mf n_\varphi$.
\end{enumerate}
\end{proposition}

\begin{proposition}
Let $M$ be a von Neumann algebra with a nsf weight $\varphi$ on $M$, with GNS
construction $(H,\Lambda,\pi)$, and modular automorphism group $(\sigma_t)$.
Let $\mf A_0\subseteq H$ be the Tomita algebra, with modular automorphism group
$(\sigma_t^0)$ and representation $\pi_L:\mf A_0\rightarrow M$, so that
$\pi_L\circ\sigma_t^0 = \sigma_t \circ \pi_L$.  Then $\pi_L(\mf A_0)$
is a core for $\sigma_{z}$ on $M$.
\end{proposition}
\begin{proof}
As $\pi_L(\mf A_0)$ generates $M$, this follows immediately from
Theorem~\ref{thm:thm1}.
\end{proof}

As an illustration of the utility of the ideas developed and summarised so far, we now
wish to give a short proof of the results of \cite{bcm}, where careful calculation
with functional calculus and unbounded operator techniques were used.
We will abstract the setting of
\cite{bcm} away from Markov operators, and instead work in the following setting:
we have Hilbert spaces $H,K$ and positive non-degenerate
(unbounded) operators $\Delta_H, \Delta_K$ on $H$ and $K$ respectively.  Thus
$(\Delta_H^{it})_{t\in\mathbb R}$ is a one-parameter (strongly-continuous) unitary group
on $H$, and similarly for $(\Delta_K^{it})$ on $K$.  Suppose we have a bounded operator
$T:H\rightarrow K$ with $T \Delta_H^{it} = \Delta_K^{it} T$ for all $t\in\mathbb R$.

\begin{proposition}[{\cite[Theorem~1.1]{bcm}}]
With the above setup, we have that $\Delta^{-t}_K T \Delta^t_H$ is densely-defined,
and bounded, with closure $T$, for each $t\in\mathbb R$.
\end{proposition}
\begin{proof}
This follows almost immediately from \cite[Theorem~6.2]{cz}, compare
Theorem~\ref{thm:three}.  Indeed, we define a weak$^*$-continuous one-parameter isometry
group on $\mc B(H,K)$ by $\alpha_t(x) = \Delta_K^{-it} x \Delta_H^{it}$.
The hypothesis on $T$ is precisely that $\alpha_t(T) = T$ for all $t$, and from this it
follows that $T$ is analytic for $(\alpha_t)$ and $\alpha_z(T)=T$ for all $z$.
In particular, $T\in D(\alpha_{-it})$ with $\alpha_{-it}(T)=T$, so from
\cite[Theorem~6.2]{cz}, it follows that $D(\Delta^{-t}_K T \Delta^t_H) = D(\Delta^t_H)$
and $\Delta^{-t}_K T \Delta^t_H$ (which is thus densely-defined) has bounded closure
equal to $T$, as required.
\end{proof}

Along the way, \cite{bcm} proves more, and in particular \cite[Theorem~3.1]{bcm}, in
our more abstract setting, is the following result, which we think is interesting in
its own right.

\begin{theorem}
With the above setup, for any $z\in\mathbb C$, we have that $T\Delta_H^z$ is
closeable, with closure $\Delta_K^z T$.
\end{theorem}
\begin{proof}
From \cite[Theorem~6.2]{cz} (as applied in the above proof), we know that
$D(\Delta_K^{-z} T \Delta_H^z) = D(\Delta_H^z)$ and $\Delta_K^{-z} T \Delta_H^z
\subseteq T$.  Equivalently, that for $\xi\in D(\Delta_H^z)$, we have that
$T \Delta_H^z\xi \in D(\Delta_K^{-z})$ with $\Delta_K^{-z} T \Delta_H^z\xi = T\xi$.
Equivalently, $T \Delta_H^z\subseteq \Delta_K^z T$.  As $\Delta_K^z T$ is always
closed, we have in fact that $\overline{T \Delta_H^z}\subseteq \Delta_K^z T$.

Consider the unitary group $(\Delta_K^{it})$ and for $n>0$ form $R_K=\mc R_n$ on $K$.
Similarly form $R_H$.  As $T \Delta_H^{it} = \Delta_K^{it} T$, it follows that
$R_K T = T R_H$.  Furthermore, for any $\xi\in H$,
\begin{align*}
T \Delta_H^z R_H\xi &= \frac{n}{\sqrt\pi} \int_{\mathbb R}
e^{-n^2(t-z)^2} T \Delta_H^{it} \xi \ dt \\
&= \frac{n}{\sqrt\pi} \int_{\mathbb R}
e^{-n^2(t-z)^2} \Delta_K^{it} T\xi \ dt
= \Delta_K^z R_K T\xi,
\end{align*}

Let $(\xi,\eta) \in \mc G(\Delta_K^z T)$, that is, $(T\xi,\eta)\in\mc G(\Delta_K^z)$.
Thus $(TR_H\xi,R_K\eta) = (R_KT\xi, R_K\eta)\in\mc G(\Delta_K^z)$, that is,
$\Delta_K^z R_K T\xi = R_K\eta$, but from above, $\Delta_K^z R_K T\xi
= T \Delta_H^z R_H \xi$, 
and so $(R_H\xi, R_K\eta) = (R_H\xi, T \Delta_H^z R_H \xi) \in \mc G(T\Delta_H^z)
\subseteq \mc G(\overline{T\Delta_H^z})$.  By Lemma~\ref{lem:two} we
see that $(\xi,\eta)$ is in the norm-closed $(\Delta^{it}_H, \Delta^{it}_K)$-invariant
subspace of $H\oplus K$ generated by $(R_H\xi, R_K\eta)$.  However, this shows that
$(\xi,\eta) \in \mc G(\overline{T\Delta_H^z})$.  In turn, this shows that
$\Delta_K^z T \subseteq \overline{T\Delta_H^z}$.

Hence we have that $\Delta_K^z T = \overline{T\Delta_H^z}$ as claimed.
\end{proof}

The setup of \cite{bcm} is actually as follows: $\Phi:(N,\rho)\rightarrow (M,\varphi)$
is a $(\rho,\varphi)$-Markov map and $T$ is defined by $Tx\xi_\rho = \Phi(x)\xi_\varphi$
for $x\in N$.  The second part of \cite[Theorem~1.1]{bcm} shows that $T$ also intertwines
the modular conjugations $J_\rho$ and $J_\varphi$.  This follows readily, as
\[ J_\varphi T J_\rho x \xi_\rho
= J_\varphi T \sigma^\rho_{i/2}(x)^* \xi_\rho
= J_\varphi \Phi(\sigma^\rho_{i/2}(x)^*) \xi_\rho. \]
As $\Phi \sigma^\rho_t = \sigma^\varphi_t \Phi$ for each $t\in\mathbb R$, an analytic
continuation argument shows that $\Phi \sigma^\rho_{i/2} \subseteq
\sigma^\varphi_{i/2} \Phi$.  Thus
\[ J_\varphi T J_\rho x \xi_\rho
= J_\varphi \sigma^\varphi_{i/2}(\Phi(x)^*) \xi_\rho
= \Phi(x) \xi_\rho = T x \xi_\rho. \]
Thus $J_\varphi T J_\rho = T$.

\subsection{The graph as a Banach algebra}\label{sec:graph_ba}

When $A$ is a Banach algebra and $(\alpha_t)$ a one-parameter automorphism group, we have
seen that $\mc G(\alpha_z)\subseteq A\oplus A$ is a (closed) subalgebra; and when
$A$ is a dual Banach algebra and $(\alpha_t)$ weak$^*$-continuous, then $\mc G(\alpha_z)$
is also a dual Banach algebra.

If further $A$ is a Banach $*$-algebra, then let us consider $\mc G(\alpha_{-i})$
(here any member of $i\mathbb R$ would lead to similar conclusions).  Given
$a\in D(\alpha_{-i})$, by Proposition~\ref{prop:star}, we have
that $a^* \in D(\alpha_i)$ with $\alpha_i(a^*) = \alpha_{-i}(a)^*$.  In particular,
$\alpha_{-i}(a)^* \in D(\alpha_{-i})$ with $\alpha_{-i}(\alpha_{-i}(a)^*) = a^*$.
It follows that $(a,b)\in\mc G(\alpha_{-i})$ if and only if $(b^*,a^*)\in
\mc G(\alpha_{-i})$.  For $a\in D(\alpha_{-i})$ write $a^\natural = \alpha_{-i}(a)^*$.
Thus $\mc G(\alpha_{-i})$ becomes a Banach $*$-algebra.  Similar considerations apply
to the dual Banach algebra case.

To our knowledge, there has been little systematic study of these Banach $*$-algebras.
There is an intriguing result stated without proof in \cite{z2}, which gives a
characterisation of which algebras $\mc G(\alpha_{-i})$ can arise, in the case of
a weak$^*$-continuous one-parameter automorphism group $(\alpha_t)$ on a von Neumann
algebra $M$.  In particular, for $x\in M$ invertible, there is a (unique) unitary
$u\in M$ such that $ux, (ux)^{-1}$ are both member of $D(\alpha_{-i})$ with
$\alpha_{-i/2}(ux), \alpha_{-i/2}((ux)^{-1})$ both positive.  A proof of this
factorisation result may be found in \cite[Section~3]{z1}, which in turn uses
ideas from \cite{w1}.

\begin{example}\label{ex:five}
Let $A=C_0(\mathbb R)$ and let $(\alpha_t)$ be the ``translation group'' defined by
$\alpha_t(f)(s) = f(s-t)$ for $s,t\in\mathbb R, f\in C_0(\mathbb R)$.  Suppose that
$f\in D(\alpha_{-i})$ with analytic extension $F:S(-i)\rightarrow C_0(\mathbb R)$.
Define $g:S(i)\rightarrow\mathbb C$ by $g(w) = F(-w)(0)$ so that $g(t) = F(-t)(0)
= \alpha_{-t}(f)(0) = f(0-(-t)) = f(t)$ for $t\in\mathbb R$.  Thus $g$ is (scalar-valued)
regular and extends $f$, and $F(-i)(t) = \alpha_{-t}(F(-i))(0) = F(-i-t)(0)
= g(i+t)$ so $\alpha_{-i}(f) = (g(i+t))_{t\in\mathbb R}$.  As $F$ is continuous, $g$ must
satisfying the ``uniformly in $C_0$ condition'' that, for $\epsilon>0$, there is $K>0$
so that $|g(x+iy)|<\epsilon$ if $|x|>K$ (for any $0\leq y\leq 1$).

Conversely, suppose that $f\in C_0(\mathbb R)$ admits such an extension $g$ to $S(i)$
(so $g$ is ``uniformly in $C_0$'').
Define $F:S(-i)\rightarrow C_0(\mathbb R)$ by $F(w) = (g(-w+t))_{t\in\mathbb R}$.
Then each $F(w)\in C_0(\mathbb R)$ and $F(t) = (f(-t+s))_{s\in\mathbb R}
= \alpha_t(f)$.  Furthermore, $F$ is norm-continuous (from the condition on $g$).
To show that $F$ is analytic on the interior of $S(-i)$, we need only show that
$\mu\circ F$ is (scalar) analytic for each $\mu\in X$ where $X\subseteq C_0(\mathbb R)^*$
is any norming subspace.  If we take $X$ to be the closed span of the point-mass measures
(so $X=\ell^1(\mathbb R)$) this follows immediately from $g$ being analytic.
Thus $F$ is regular and so $f\in D(\alpha_{-i})$ with $\alpha_{-i}(f)(t) = F(-i)(t)
= g(t+i)$ for $t\in\mathbb R$.

Thus $\mc G(\alpha_{-i})$ may be identified with a space of scalar-valued regular
functions on the strip $S(i)$, which we could think of as some sort of
``generalised Hardy space''.

Similarly, $(\alpha_t)$ extends to a weak$^*$-continuous automorphism group on
$L^\infty(\mathbb R)$.  A slightly more involved argument, making use of the smearing
technique, similarly allows us to regard $\mc G(\alpha_{-i})$ as being the subspace of
$L^\infty(S(i))$ consisting of functions analytic on the interior of $S(i)$,
and having suitable boundary values.
\end{example}

There are related Banach algebras which have been more studied.  We first quickly
recall Arveson's notion of spectral subspace from \cite[Section~2]{arv1}.  In our
setting, these are studied in \cite[Section~5]{cz} and \cite{z}, see in particular
the comment at the bottom on page~86 in \cite{z}.  These are subspaces of elements
which are analytic for $(\alpha_t)$, and which have certain growth rates at infinity.

To be more precise, for example,
let $M$ be a von Neumann algebra and $(\alpha_t)$ a weak$^*$-continuous automorphism
group of $M$.  In particular, following \cite{z}, we define $M^\alpha([1,\infty))$
to be the collection of $x\in M$ such that $x\in D(\alpha_{in})$ for $n=1,2,\cdots$
and $\limsup_n \|\alpha_{in}(x)\|^{1/n}\leq 1$.  This space is often denoted by
$H^\infty(\alpha)$; indeed, it is shown in \cite[Proposition~2.1]{kt} that
$x\in M^\alpha([1,\infty))$ if and only if $t\mapsto \ip{\alpha_{-t}(x)}{\omega}$
is in $H^\infty(\mathbb R)$, for each $\omega\in M_*$.  We say that $M$ is
\emph{$\alpha$-finite} when the collection $\{ \omega\in M_*^+ : \omega\circ\alpha_t
= \omega \ (t\in\mathbb R) \}$ separates the points of $M_+$.  In this case,
$H^\infty(\alpha)$ is a \emph{maximal subdiagonal algebra} in the sense of
\cite{arv3}.  For more on this topic, see \cite{j, lm} for example.  
Maximal subdiagonal algebras have been widely studied as non-commutative analogues of
Hardy spaces.

\begin{example}
Consider $L^\infty(\mathbb R)$ with the shift group, as in Examples~\ref{ex:five}.
Then $H^\infty(\alpha)$ is simply the classical Hardy space of the upper-half-plane
$H^\infty(\mathbb R)$, see \cite[Introduction]{lm}.

As in Example~\ref{ex:four}, let $P(\delta_n) = p_n\delta_n$ on
$\ell^2 = \ell^2(\mathbb N)$, and define $\alpha_t(x) = P^{it} x P^{-it}$ for
$x\in\mc B(\ell^2)$.  Consider the matrix unit $e_{jk}$ which sends $\delta_k$ to
$\delta_j$.  Then $\alpha_t(e_{jk}) = p_k^{-it} p_j^{it} e_{jk}$ and so
$\alpha_{in}(e_{jk}) = p_k^{n} p_j^{-n} e_{jk}$ for each $n=1,2,\cdots$.
It follows that $e_{jk} \in H^\infty(\alpha)$ if and only if $p_k / p_j \leq 1$.
If $(p_n)$ is an increasing sequence, then $e_{jk} \in H^\infty(\alpha)$ exactly when
$k\leq j$.  A more involved calculation shows that $H^\infty(\alpha)$ consists
exactly of the lower-triangular matrices.
\end{example}

While $\mc G(\alpha_{-i})$ is clearly different from $H^\infty(\alpha)$, there are some
intriguing similarities.  For example, the factorisation result of Zsido mentioned
above, \cite{z2}, is very similar to Arveson's factorisation result,
\cite[Section~4.2]{arv3}, showing that if $x\in M$ is invertible then there is
$a\in H^\infty(\alpha)$ with $a^{-1}\in H^\infty(\alpha)$, and a unitary $u\in M$,
with $x=ua$.  We wonder if there is further to be developed here; in particular,
is there a notion of $L^p$ space for $\mc G(\alpha_{-i})$, similar to that for
subdiagonal algebras, compare \cite{mw}?

\begin{remark}\label{rem:5}
Consider $\mc A = \mc G(\alpha_{-i})$ as a weak$^*$-closed subalgebra of
$M\oplus_\infty M$.  Let $(x,y)\in \mc A \cap \mc A^*$ so $(x^*,y^*)\in\mc A$
so $(y,x)\in\mc A$ (given the above remarks).  There are hence weak$^*$-regular maps
$f,g:S(-i)\rightarrow M$ with $f(t)=\alpha_t(x), g(t)=\alpha_t(y)$ for $t\in\mathbb R$
and $f(-i)=y, g(-i)=x$.  It follows that $f(t-i) = g(t)$ and so ``glueing'' these
maps together we obtain $h:S(-2i)\rightarrow M$ which by Morera's Theorem is regular,
has $h(t)=\alpha_t(x)$ and $h(t-2i) = g(-i)=x = h(t)$ for $t\in\mathbb R$.  By ``tiling''
we can extend $h$ to an entire map on $\mathbb C$ which is bounded, and hence constant.
This shows that $x=y$ and $\alpha_t(x)=x$ for all $t$.  We conclude that $\mc A\cap
\mc A^* = \{ (x,x) : x\text{ is }(\alpha_t)\text{-invariant}\}$.

Now consider when $\mc A+\mc A^*$ is weak$^*$-dense in $M\oplus M$.  If
$(\omega,\tau)\in {}^\perp(\mc A+\mc A^*)$ then $(\omega,\tau)\in {}^\perp\mc A
\cap {}^\perp\mc A^*$ so $(-\tau,\omega), (-\tau^*,\omega^*)\in\mc G(\alpha^{M_*}_{-i})$.
Arguing as in the previous paragraph, this is if and only if $\omega=-\tau$ is
$(\alpha^{M_*}_t)$-invariant.  Now, $\omega$ is $(\alpha^{M_*}_t)$-invariant if and only
if $\omega\in{}^\perp X$ where $X$ is the weak$^*$-closed linear span of
$\{ x-\alpha_t(x) : x\in M, t\in\mathbb R \}$.  It follows that $(x,y)$ is in the
weak$^*$-closure of $\mc A+\mc A^*$ if and only if $\ip{(x,y)}{(\omega,-\omega)}=0$
for each $\omega\in {}^\perp X$, that is, $x-y\in ({}^\perp X)^\perp = X$.

For $\mc A$ to be a (finite, maximal) subdiagonal algebra of $M\oplus M$ we would want
that $\mc A\cap\mc A^*$ to be the range of a faithful normal conditional expectation,
and we'd want $X$ to be all of $M$ (equivalently, there to be no non-zero
$(\alpha_t^{M_*})$-invariant functionals).  If $(\alpha_t)$ is trivial, then this
is obviously not the case.  For the shift-group on $L^\infty(\mathbb R)$, however,
we do have that $\mc A+\mc A^*$ is weak$^*$-dense in $M\oplus M$, and $\mc A\cap\mc A^*$
is $\mathbb C(1,1)$, but there are no \emph{normal} conditional
expectations $M\oplus M$ to $\mathbb C(1,1)$ which are \emph{multiplicative} on
$\mc G(\alpha_{-i})$.
\end{remark}

We finish this section with one general Banach algebraic result.

\begin{proposition}\label{prop:19}
Let $A$ be a Banach algebra with a bounded approximate identity bounded by $M\geq 1$.
Let $(\alpha_t)$ be a (norm-continuous) automorphism group on $A$.
For any $z$ we have that $\mc G(\alpha_z)$ has a bounded approximate identity
bounded by $M\geq 1$.
\end{proposition}
\begin{proof}
We just give a sketch, as this could be proved exactly as \cite[Proposition~2.26]{kus1}
(which is attributed to Van Daele and Verding); compare also the proof of
\cite[Theorem~12]{ds}.
Indeed, as $A$ has a bounded approximate identity, it admits a theory of multiplier
algebras paralleling that of $C^*$-algebras.  Instead of developing this theory,
we give a direct proof.

Let $(e_i)$ be a bounded approximate identity with $\|e_i\|\leq M$ for each $i$.
The key idea is to consider $\mc R_n(e_i)$ with $n>0$ \emph{small} and not large.
This will ensure that $\|\alpha_z\mc R_n(e_i)\|$ will be close to $M$.  For $a\in A$,
as $t\mapsto \alpha_t(a)$ is norm-continuous, for any $K>0$ the set $\{ \alpha_{-t}(a):
|t|\leq K \}$ is compact, and so $e_i \alpha_{-t}(a)\rightarrow \alpha_{-t}(a)$
uniformly for $|t|\leq K$.  It follows that $\alpha_t(e_i) a
= \alpha_t( e_i\alpha_{-t}(a) ) \rightarrow \alpha_t(\alpha_{-t}(a)) = a$ uniformly
for $|t|\leq K$.  By the integral form of $\mc R_n$ and $\alpha_z\mc R_n$, it follows
that if $i$ is sufficiently large, then $\|\mc R_n(e_i)a-a\|$ and 
$\|\alpha_z(\mc R_n(e_i))a-a\|$ will be small.  In this way, we can construct a
bounded approximate identity in $\mc G(\alpha_z)$ with the required bound.
\end{proof}

\section{A Kaplansky Density type result}\label{sec:kap}

We again consider the case of a $C^*$-algebra generating a von Neumann algebra $M$,
with a one-parameter automorphism group on $M$ restricting to $A$.  The Kaplansky
Density Theorem tells us that the unit ball of $A$ is weak$^*$-dense in the unit ball
of $M$.  This section is devoted to proving the following; recall that
Proposition~\ref{prop:three} shows that $\mc G(\alpha^A_z)$ is weak$^*$-dense in
$\mc G(\alpha^M_z)$.

\begin{theorem}\label{thm:six}
With $A, M, (\alpha_t)$ as before, let $z\in\mathbb C$, let $\alpha^A_z$ be the
analytic extension on $A$, and $\alpha^M_z$ that on $M$.  In $M\oplus_\infty M$,
or $M\oplus_1 M$, the unit ball of $\mc G(\alpha^A_z)$ is weak$^*$-dense in the
unit ball of $\mc G(\alpha^M_z)$.
\end{theorem}

Let $M_*$ be the predual of $M$.  By restricting functionals in $M_*$ to $A\subseteq M$,
we define a map $\iota:M_*\rightarrow A^*$.  By Kaplansky Density, this map is an
isometry.  It is easy to see that it preserves the $A$-module actions, and so $M_*$ is
identified with a closed $A$-subbimodule of $A^*$.
By \cite[Section~2, Chapter~III]{tak1} there is a central
projection $p\in A^{**}$ with $pA^* = A^*p = M_*$.  In fact, we construct $p$ by noticing
that $M_*^\perp = \{ x\in A^{**} : \ip{x}{\omega}=0 \ (\omega\in M_*\subseteq A^*)\}$
is a weak$^*$-closed ideal in $A^{**}$ and so $M_*^\perp = A^{**}p'$ for some central
projection $p'\in A^{**}$; we then set $p=1-p'$.  We furthermore have that
\[ A^* \cong pA^* \oplus_1 (1-p)A^*, \qquad
A^{**} \cong pA^{**} \oplus_\infty (1-p)A^{**}. \]

\begin{lemma}\label{lem:six}
Let $\beta$ be a $*$-automorphism of $A$, and suppose that $\beta^*(M_*)\subseteq M_*$.
Then $\beta^{**}(p) = p$.  In particular, $\alpha_t^{**}(p)=p$ for all $t$.
\end{lemma}
\begin{proof}
We have that $\beta^{**}$ is a $*$-automorphism of $A^{**}$, and so $q = \beta^{**}(p)$
is a central projection.  Then $(1-q)A^{**} = \beta^{**}((1-p)A^{**}) =
\beta^{**}(M_*^\perp) = M_*^\perp$ as $\beta^*(M_*) = M_*$.  Thus $(1-q)A^{**} 
= M_*^\perp$ and so $q=p$.
\end{proof}

In the following lemma, we identify $A$ with a subspace of $A^{**}$ in the canonical way.

\begin{lemma}\label{lem:seven}
For $a\in A$ and $s,t\in\mathbb R$ we have that $\alpha_s^{**}(p\alpha_t(a))
= p \alpha_{s+t}(a)$.
\end{lemma}
\begin{proof}
As $\alpha_s^{**}$ is an automorphism, and using Lemma~\ref{lem:six}, we have that
$\alpha_s^{**}(p\alpha_t(a)) = p \alpha_s^{**}(\alpha_t(a))$.  A simple calculation
shows that for $b\in A$, we have that $\alpha_s^{**}(b)$ is equal to the image of
$\alpha_s(b)\in A$ in $A^{**}$.  The result follows.
\end{proof}

To easy notation, fix $z\in\mathbb C$ and let
$\mc G = \mc G(\alpha_z^A)$ regarded as a subspace of $A\oplus_\infty A$.
Similarly let $\mc G^M = \mc G(\alpha_z^M)$ regarded as a subspace of
$M\oplus_\infty M$.
Notice that the dual space of $A\oplus_\infty A$ is $A^*\oplus_1 A^*$, and the
bidual is $A^{**}\oplus_\infty A^{**}$.  Then $(p,p)$ is a central projection in
$A^{**}\oplus_\infty A^{**}$.  By the Hanh-Banach theorem, we can identify the dual space
of $\mc G$ with the quotient $(A^*\oplus_1 A^*) / \mc G^\perp$, and in turn identify
the dual of this quotient with $\mc G^{\perp\perp}$.
Thus $\mc G^{**} = \mc G^{\perp\perp}$.

\begin{theorem}\label{thm:five}
We have that $(p,p) \mc G^{\perp\perp} \subseteq \mc G^{\perp\perp}
\subseteq A^{**}\oplus A^{**}$.
\end{theorem}
\begin{proof}
Let $a\in A$, let $n>0$, and define $f:S(z)\rightarrow A^{**}$ by
\[ f(w) = \frac{n}{\sqrt\pi} \int_{\mathbb R}
\exp(-n^2(t-w)^2) p\alpha_t(a) \ dt. \]
As $t\mapsto \alpha_t(a)$ is norm-continuous, also $t\mapsto p\alpha_t(a)$ is
norm-continuous, and so the integral defining $f$ is norm convergent, and
$f$ is norm-regular.  In fact, we have that $f(w) = p \alpha_w(\mc R_n(a))$.
From Lemma~\ref{lem:seven}, we have that $\alpha_s^{**}(f(w))
= f(w+s)$ for $w\in S(z), s\in\mathbb R$.

Let $(-\lambda,\mu)\in \mc G^\perp$, which is equivalent to $\mu\in D(\alpha^{A^*}_z)$
with $\alpha^{A^*}_z(\mu)=\lambda$.  Thus there is $g:S(z)\rightarrow A^*$ a
weak$^*$-regular function with $g(t)=\alpha_t^*(\mu)$ for each $t\in\mathbb R$, and with
$g(z)=\lambda$.

Define $h:S(z)\rightarrow\mathbb C$ by $h(w) = \ip{f(w)}{g(z-w)}$.  Then
\[ h(t) = \ip{f(t)}{g(z-t)} = \ip{\alpha_t^{**}(f(0))}{\alpha_{-t}^*(\lambda)}
= \ip{f(0)}{\lambda} \qquad (t\in\mathbb R). \]
Thus $h$ is constant on $\mathbb R$.  Furthermore, for $w\in S(z)$,
\[ h(w) = \frac{n}{\sqrt\pi} \int_{\mathbb R} \exp(-n^2(t-w)^2)
\ip{p\alpha_t(a)}{g(z-w)} \ dt, \]
here again using that the integral defining $f$ is norm-convergent.
Now, $\ip{p\alpha_t(a)}{g(z-w)} = \ip{pg(z-w)}{\alpha_t(a)}$, and so
\[ h(w) = \frac{n}{\sqrt\pi} \int_{\mathbb R} \exp(-n^2(t-w)^2)
\ip{pg(z-w)}{\alpha_t(a)} \ dt
= \ip{pg(z-w)}{\alpha_w(\mc R_n(a))}. \]
As $w\mapsto \alpha_w(\mc R_n(a))$ is a norm-continuous map, and $w\mapsto
pg(z-w)$ is bounded and weak$^*$-continuous, it follows that $h$ is continuous on $S(z)$.
On the interior of $S(z)$, we have that $h$ is the pairing between two functions given
locally by power series.  We conclude that $h$ is regular.  As $h$ is constant on
$\mathbb R$, $h$ must be constant on $S(z)$.

Thus
\begin{align*}
\ip{(p,p)(\mc R_n(a), \alpha_z(\mc R_n(a)))}{(-\lambda,\mu)}
&= \ip{-p\lambda}{\mc R_n(a)} + \ip{p\mu}{\alpha_z(\mc R_n(a))} \\
&= - \ip{f(0)}{\lambda} + \ip{f(z)}{\mu} = -h(0) + h(z) = 0. \end{align*}
By Theorem~\ref{thm:thm1},
$\{ (\mc R_n(a), \alpha_z(\mc R_n(a))) : a\in A \}$ is norm dense in $\mc G$, and as
$(-\lambda,\mu)\in\mc G^\perp$ was arbitrary, the above calculation shows that
$(p,p)\mc G \subseteq \mc G^{\perp\perp}$.  By weak$^*$-continuity, we conclude that
$(p,p)\mc G^{\perp\perp} \subseteq \mc G^{\perp\perp}$ as claimed.
\end{proof}

\begin{lemma}\label{lem:eight}
Let $\mf A$ be a dual Banach algebra, let $X\subseteq\mf A$ be a weak$^*$-closed subspace,
let $p\in \mf A$ be an idempotent (so $p^2=p$) and suppose that $pX\subseteq X$.
Then $pX$ is weak$^*$-closed.
\end{lemma}
\begin{proof}
Let $(x_i)$ be a net in $X$ with $px_i\rightarrow a\in \mf A$ weak$^*$.
We aim to show that $a\in pX$.  Now, $px_i = p^2x_i \rightarrow pa$ as $\mf A$ is a dual
Banach algebra.  Thus $a=pa$.  Now, also $px_i \in pX \subseteq X$, by hypothesis, and
as $X$ is weak$^*$-closed, $a\in X$.  Thus $a = pa \in pX$ as required.
\end{proof}

As above, as $A\subseteq M$, restriction of functionals gives
$\iota:M_*\rightarrow A^*$, which is an isometric inclusion by Kaplansky density.
Furthermore, we have that $pA^* = \iota(M_*)$, and so we have an
inverse map $\iota^{-1} : pA^* \rightarrow M_*$ and so the Banach space adjoint
is a map $(\iota^{-1})^*:M\rightarrow (pA^{*})^* \cong pA^{**}$.  We give a word of
warning: the composition of the isometries $A\rightarrow M \cong pA^{**} \rightarrow
A^{**}$ is \emph{not} the canonical map $A\rightarrow A^{**}$, but is rather the map
$a \mapsto pa\in A^{**}$.

\begin{lemma}\label{lem:nine}
Identifying $M\oplus M$ with $pA^{**}\oplus pA^{**}$, and regarding $\mc G$ as a
subspace of $A^{**}\oplus A^{**}$ in the canonical way, we have that
$(p,p)\mc G \subseteq \mc G^M$.
\end{lemma}
\begin{proof}
Denote by $\phi$ the corestriction of $\iota$, so
$\phi$ is an isometric isomorphism $M_*\rightarrow pA^*$, and hence $\phi^*:pA^{**}
\rightarrow M$ is an isomorphism, the inverse of $(\iota^{-1})^*$.  Similarly, let
$\psi:A\rightarrow M$ be the inclusion.  For $a\in A$ and $\omega\in M_*$,
\begin{align*} \ip{\phi^*(pa)}{\omega} &= \ip{pa}{\phi(\omega)}
= \ip{p\phi(\omega)}{a} = \ip{p\iota(\omega)}{a} 
= \ip{\iota(\omega)}{a} = \ip{\psi(a)}{\omega}. \end{align*}
This $\phi^*(pa) = \psi(a)$.  As $(\psi\oplus\psi)\mc G \subseteq \mc G^M$,
the result follows.
\end{proof}

\begin{theorem}
Identifying $M\oplus M$ with $pA^{**}\oplus pA^{**}$, we have that
$(p,p)\mc G^{\perp\perp} = \mc G^M$.
\end{theorem}
\begin{proof}
Given $(\alpha,\beta)\in\mc G^{\perp\perp}$ there is a bounded net $(a_i,b_i)$ in $\mc G$
converging weak$^*$ to $\alpha$.  Then $(p\alpha,p\beta)$ is the weak$^*$-limit of the
net $(pa_i, pb_i)$, and by Lemma~\ref{lem:nine} we know that $pa_i\in\mc G^M$
for each $i$.  As $\mc G^M \subseteq pA^{**}\oplus pA^{**}$ is weak$^*$-closed, we
conclude that $(p,p)\mc G^{\perp\perp} \subseteq \mc G^M$.

We apply Lemma~\ref{lem:eight} to $A^{**}\oplus A^{**}$ and the idempotent $(p,p)$,
with the subspace $\mc G^{\perp\perp}$.  By Theorem~\ref{thm:five}, the hypothesis of
Lemma~\ref{lem:eight} holds, and so $(p,p)\mc G^{\perp\perp}$ is weak$^*$-closed.

Given $(x,y)\in\mc G^M$, by Proposition~\ref{prop:three}, there is a net (perhaps
\emph{not} bounded) $(a_i,b_i)$ in $\mc G$ converging weak$^*$ to $(x,y)$ in $M\oplus M$.
As $M\rightarrow pA^{**}$ is weak$^*$-continuous, it follows that the net
$(pa_i,pb_i)$ converges weak$^*$ to $(x,y)\in \mc G^M\subseteq pA^{**}\oplus pA^{**}$.
This net is in $(p,p)\mc G \subseteq (p,p)\mc G^{\perp\perp}$, and as
$(p,p)\mc G^{\perp\perp}$ is weak$^*$-closed, we conclude that
$(x,y)\in (p,p)\mc G^{\perp\perp}$.  Thus $\mc G^M \subseteq (p,p)\mc G^{\perp\perp}$
and we have equality.
\end{proof}

Our main theorem now follows easily.

\begin{proof}[Proof of Theorem~\ref{thm:six}]
Given a member of the unit ball of $\mc G^M$, we regard $\mc G^M$ as being
$(p,p)\mc G^{\perp\perp} \subseteq \mc G^{\perp\perp}$, and so we have a member of the
unit ball of $\mc G^{\perp\perp} = \mc G^{**}$.
By Hahn-Banach (that is, the Goldstine theorem)
there is a net in the unit ball of $\mc G$ converging weak$^*$ to our element of
$\mc G^M$, as we want.

To deal with the $\oplus_1$ normed case, we simply follow the same proof through,
using $pA^*\oplus_\infty (1-p)A^*$ and $pA^{**}\oplus_1 (1-p)A^{**}$.  While
$pA^{**}\oplus_1 (1-p)A^{**}$ is not a $C^*$-algebra, it is still a Banach algebra,
and so Lemma~\ref{lem:eight} still holds, and the rest follows.
\end{proof}

We finish with a result about stronger topologies.

\begin{corollary}
With the hypotheses of Theorem~\ref{thm:six}, the unit ball of $\mc G(\alpha^A_z)$,
in $M\oplus_\infty M$, is $\sigma$-strong$^*$-dense in the unit ball of
$\mc G(\alpha_z^M)$.
\end{corollary}
\begin{proof}
This follows immediately from the general result \cite[Theorem~2.6(iv)]{tak1} that
in a von Neumann algebra $N$, for a convex subset $K$ we have that the weak$^*$ and
$\sigma$-strong$^*$ closures of $K$ agree.
\end{proof}

\section{Duals of automorphism groups}\label{sec:duals}

In this section, we shall look at the ``dual'' situation to the previous section.
We again consider the case of a $C^*$-algebra generating a von Neumann algebra $M$,
with a one-parameter automorphism group $(\alpha^M_t)$ on $M$ restricting to $A$, say
to given $(\alpha^A_t)$.  Then the preadjoint gives a (norm-continuous) one-parameter
isometry group $(\alpha^{M_*}_t)$ on $M_*$, and the adjoint gives a (weak$^*$-continuous)
one-parameter isometry group $(\alpha^{A^*}_t)$ on $A^*$.  A simple calculation shows
that the inclusion $\iota:M_*\rightarrow A^*$ intertwines these groups.

\begin{proposition}\label{prop:20}
For $z\in\mathbb C$, we have that $D(\alpha^{M_*}_z)$ is a
(weak$^*$) core for $D(\alpha^{A^*}_z)$.
\end{proposition}
\begin{proof}
Follows exactly as the proof of Proposition~\ref{prop:three}.
\end{proof}

The following is a stronger version of Proposition~\ref{prop:4}.
We recall that the analogous result for the inclusion $A\rightarrow M$ is false,
see Example~\ref{ex:four}.

\begin{theorem}\label{thm:seven}
Let $\omega\in M_*$ be such that $\iota(\omega)\in D(\alpha^{A^*}_z)$.  Then
$\omega\in D(\alpha^{M_*}_z)$.
\end{theorem}
\begin{proof}
We continue with the notations of the previous section.  Theorem~\ref{thm:five}
shows that $(p,p)\mc G^{\perp\perp} \subseteq \mc G^{\perp\perp}$ and so any
$(\alpha,\beta)\in\mc G^{\perp\perp}$ is equal to
\[ (p\alpha, p\beta) + ((1-p)\alpha, (1-p)\beta), \]
where both summands are members of $\mc G^{\perp\perp}$.

Let $(\mu,\lambda)\in \mc G(\alpha^{A^*}_z)$, equivalently, $(-\lambda,\mu)
\in \mc G^\perp$, equivalently, $\ip{(\alpha,\beta)}{(-\lambda,\mu)} = 0$ for all
$(\alpha,\beta)\in\mc G^{\perp\perp}$.  Given the above discussion, this in turn is
equivalent to
\[ \ip{(p\alpha,p\beta)}{(-\lambda,\mu)} =
\ip{((1-p)\alpha,(1-p)\beta)}{(-\lambda,\mu)} = 0
\qquad ( (\alpha,\beta)\in\mc G^{\perp\perp} ). \]
That is,
\[ \ip{(\alpha,\beta)}{(-p\lambda,p\mu)} =
\ip{(\alpha,\beta)}{(-(1-p)\lambda,(1-p)\mu)} = 0
\qquad ( (\alpha,\beta)\in\mc G^{\perp\perp} ). \]
Reversing this argument shows that $(\mu,\lambda)\in \mc G(\alpha^{A^*}_z)$
if and only if both $(p\mu,p\lambda)\in \mc G(\alpha^{A^*}_z)$ and
$((1-p)\mu,(1-p)\lambda)\in \mc G(\alpha^{A^*}_z)$.

In particular, if $(\mu,\lambda)\in \mc G(\alpha^{A^*}_z)$ with $\mu\in M_*$, that is,
$p\mu=\mu$, then $(\mu,p\lambda) \in \mc G(\alpha^{A^*}_z)$, but as this is a graph,
it follows that $\lambda = p\lambda$, that is, $\lambda\in M_*$.  The result now
follows as in the proof of Proposition~\ref{prop:4}.
\end{proof}

\subsection{Quotients}\label{sec:quot}

Let $E$ be a Banach space and $(\alpha_t)$ a norm-continuous one-parameter group of
isometries.  Suppose that $F\subseteq E$ is a closed subspace with $\alpha_t(F)
\subseteq F$ for each $t$.  It is easy to see that $E/F\rightarrow E/F; x+F \mapsto
\alpha_t(x)+F$ is a well-defined contraction for each $t$.  We hence obtain a
norm-continuous one-parameter group of isometries $(\alpha_t^{E/F})$ on $E/F$.
By considering analytic continuations, it is easy to see that if $(x,y)\in\mc G(\alpha_z)$
then $(x+F, y+F)\in\mc G(\alpha_z^{E/F})$.

\begin{proposition}
For any $z$ we have that $D(\alpha_z)+F \subseteq E/F$ is a core for $\alpha_z^{E/F}$.
\end{proposition}
\begin{proof}
As $D(\alpha_z)$ is dense in $E$, it follows that $D(\alpha_z)+F$ is dense in $E/F$.
As $D(\alpha_z)+F$ is also $(\alpha_t^{E/F})$-invariant, the result follows immediately
from Theorem~\ref{thm:thm1}.
\end{proof}

To say more, we consider a duality argument (that is, use the Hahn-Banach theorem).
The dual space of $\mc G(\alpha_z) \subseteq E\oplus_\infty E$ is
\[ E^*\oplus_1 E^* / \mc G(\alpha_z)^\perp
\quad\text{where}\quad
\mc G(\alpha_z)^\perp = \{ (-\lambda,\mu) : (\mu,\lambda)\in
\mc G(\alpha_z^*) \}. \]
Thus, the Banach space adjoint of the map $\mc G(\alpha_z) \rightarrow
\mc G(\alpha_z^{E/F})$ is
\[ \pi : F^\perp \oplus_1 F^\perp / \mc G(\alpha_z^{E/F})^\perp =
(E/F)^*\oplus_1 (E/F)^* / \mc G(\alpha_z^{E/F})^\perp \rightarrow
E^*\oplus_1 E^* / \mc G(\alpha_z)^\perp. \]
The proposition above implies that $\pi$ is injective.  In fact, this also follows
using the argument in the proof of Proposition~\ref{prop:4}.  Indeed, suppose that
$\lambda, \mu \in F^\perp$ with $\pi( (\mu,\lambda) + \mc G(\alpha^{E/F}_z)^\perp)=0$
so that $(\mu,\lambda) \in \mc G(\alpha_z)^\perp$, that is,
$(-\lambda, \mu) \in \mc G(\alpha_z^*)$.  As $(\alpha_t^{F^\perp})$ is the restriction of
$(\alpha_t^*)$ from $E^*$ to $F^\perp$, for all $t$, we see that $(-\lambda,\mu)\in
\mc G(\alpha^{F^\perp}_z)$.  That is, $(\mu,\lambda) \in \mc G(\alpha^{E/F}_z)^\perp$,
from which it follows that $\pi$ is injective.

However, we see no reason why $\pi$ need be bounded below, or an isometry, in
general.  We wish now to give a condition under which $\pi$ will be an isometry.

\begin{lemma}
With $E,F$ and $(\alpha_t)$ as above, suppose there is a norm-one projection
$e:E^*\rightarrow F^\perp$.  Then there is a norm-one projection
$p:E^*\rightarrow F^\perp$ with with $p \alpha_t^* = \alpha_t^* p$ for each $t$.
\end{lemma}
\begin{proof}
Consider $\mathbb R_d$, the real numbers considered as a discrete group under addition.
This group is amenable, so there is a state $\Lambda\in \ell^\infty(\mathbb R)^*$ which
is shift-invariant.  For $t\in\mathbb R$ define $e_t:E^*\rightarrow E^*$ by $e_t(\mu)
= (\alpha_{-t}^* \circ e \circ \alpha_t^*)(\mu)$.  Given $\mu\in F^\perp$, as
$\alpha_t^*(\mu) \in F^\perp$ and so $e\alpha_t^*(\mu) = \alpha_t^*(\mu)$, it follows that
$e_t(\mu)=\mu$.  Thus $e_t$ is a norm-one projection onto $F^\perp$.  For $\mu\in E^*$ and
$x\in E$, as $t\mapsto \ip{e_t(\mu)}{x}$ is bounded, the value
$\ip{\Lambda}{(\ip{e_t(\mu)}{x})}$ is well-defined.  Then
$x\mapsto \ip{\Lambda}{(\ip{e_t(\mu)}{x})}$ is linear and bounded, and so defines
$p(\mu)\in E^*$.

For $\mu\in F^\perp$ we have that $\ip{p(\mu)}{x}
= \ip{\Lambda}{(\ip{e_t(\mu)}{x})} = \ip{\Lambda}{(\ip{\mu}{x})} = \ip{\mu}{x}$
and so $p(\mu)=\mu$.  For any $\mu\in E^*$ and $x\in F$, as $\ip{e_t(\mu)}{x}=0$ for
all $t$, it follows that $p(\mu)\in F^\perp$.  Thus $p$ is a norm-one projection
$E^*\rightarrow F^\perp$.

Finally, for $s\in \mathbb R$ and arbitrary $\mu,x$ we have that
\begin{align*}
\ip{p\alpha_s^*(\mu)}{x} &= \ip{\Lambda}{(\ip{\alpha_{-t}^* e \alpha_{t+s}^*(\mu)}{x})}
= \ip{\Lambda}{(\ip{\alpha_{-(t-s)}^* e \alpha_{t}^*(\mu)}{x})} \\
&= \ip{\Lambda}{(\ip{\alpha_{-t}^* e \alpha_{t}^*(\mu)}{\alpha_s(x)})}
= \ip{p(\mu)}{\alpha_s(x)}.
\end{align*}
Thus $p\alpha_s^* = \alpha_s^*p$ as required.
\end{proof}

\begin{proposition}\label{prop:13}
With $E,F$ and $(\alpha_t)$ as above, suppose there is a norm-one projection
$p:E^*\rightarrow F^\perp$.  Then $\pi$ is an isometry, and so
$\mc G(\alpha_z) \rightarrow \mc G(\alpha_z^{E/F})$ is a metric surjection.
\end{proposition}
\begin{proof}
By the lemma, we may suppose that $p \alpha_t^* = \alpha_t^* p$ for each $t$.
Let $\mu,\lambda\in F^\perp$ with $\|(\mu,\lambda) + \mc G(\alpha_z)^\perp\| < 1$.
We aim to show that $\|(\mu,\lambda) + \mc G(\alpha_z^{E/F})^\perp\|\leq 1$.
The hypothesis is that there is $\phi\in D(\alpha_z^*)$ with
$\|\mu - \alpha_z^*(\phi)\| + \|\lambda + \phi\| < 1$.

For $n>0$ form $\mc R_n$ on $E^*$ using $(\alpha_t^*)$.
As $\mc R_n$ is norm-decreasing, we have that
$\|\mc R_n(\mu) - \mc R_n(\alpha_z^*(\phi))\| + \|\mc R_n(\lambda) + \mc R_n(\phi)\| < 1$.
Set $\phi' = \mc R_n(\phi)$, and recall that $\mc R_n(\alpha_z^*(\phi))
= \alpha_z^*(\phi')$.

By Proposition~\ref{prop:5}, we know that $t\mapsto \alpha_t^*(\phi')$ is norm-continuous,
and similarly $t\mapsto \alpha_t^*\alpha_z^*(\phi') =
\alpha_t^*(\mc R_n(\alpha_z^*(\phi)))$ is norm-continuous.
Let $f:S(z)\rightarrow E^*$ be the analytic extension of $t\mapsto \alpha_t^*(\phi')$
so $f$ is weak$^*$-regular and norm-continuous on $\mathbb R$ and $z+\mathbb R$.
By Lemma~\ref{lem:three}, it follows that $f$ is norm-regular (this could also be proved
by adapting the proof of Proposition~\ref{prop:5} to show that
$z\mapsto \alpha_z(\mc R_n(\phi))$ is norm-continuous.)
Hence also $w\mapsto p(f(w))$ is norm-regular.  It follows that $p(\phi')
\in D(\alpha_z^{F^\perp})$ with $\alpha_z^{F^\perp}(\phi') = p(f(z)) =
p(\alpha_z^*(\phi'))$.

As $p$ is a contraction, we have that
\[ \| \mc R_n(\mu) - \alpha_z^{F^\perp}(\phi'') \| +
\| \mc R_n(\lambda) + \phi'' \| < 1, \]
for $\phi'' = p(\phi') \in F^\perp$.  This shows that
$\| (\mc R_n(\mu), \mc R_n(\lambda)) + \mc G(\alpha_z^{E/F})^\perp \| < 1$, that is,
the norm of $(\mc R_n(\mu), \mc R_n(\lambda))$ in $\mc G(\alpha_z^{E/F})^*$ is at most
$1$.  As $\mc R_n(\mu)\rightarrow \mu$ weak$^*$, as $n\rightarrow\infty$, and the same
for $\lambda$, by taking weak$^*$-limits we conclude that the norm of
$(\mu,\lambda)$ in $\mc G(\alpha_z^{E/F})^*$ is at most $1$, as required.

That $\mc G(\alpha_z) \rightarrow \mc G(\alpha_z^{E/F})$ is a metric surjection
follows from the Hahn-Banach theorem.
\end{proof}

\subsection{Kaplansky-like results}\label{sec:kap_dual}

Motivated by Proposition~\ref{prop:20} and the results of Section~\ref{sec:kap},
we might wonder if the unit ball of $\mc G(\alpha_z^{M_*})$ is weak$^*$-dense in
the unit ball of $\mc G(\alpha_z^{A^*})$.  This unfortunately seems subtle, and we
can only give a partial answer.

Let us norm $\mc G(\alpha_z^{A^*})$ as a subspace of $A^* \oplus_\infty A^*$;
similar remarks would apply to other choices of norm.
Then $\mc G(\alpha_z^{A^*})$ is the dual space $A\oplus_1 A / X_A$ where
$X_A = {}^\perp \mc G(\alpha_z^{A^*}) = \{ (-b, a) : (a,b) \in \mc G(\alpha_z^{A}) \}$.
Similarly, $\mc G(\alpha_z^{M_*})^* = M \oplus_1 M / X_M$ where
$X_M = \{ (-y, x) : (x,y) \in \mc G(\alpha_z^{M}) \}$.
The Hahn-Banach theorem thus shows that the unit ball of $\mc G(\alpha_z^{M_*})$ is
weak$^*$-dense in the unit ball of $\mc G(\alpha_z^{A^*})$ if and only if
$\mc G(\alpha_z^{M_*})$ norms $A\oplus_1 A / X_A$.  This in turn is equivalent to
$A\oplus_1 A / X_A \rightarrow M\oplus_1 M / X_M$ being an isometry.

\begin{lemma}\label{lem:20}
Let $A_0\subseteq A$ be a dense subset.  The following are equivalent:
\begin{enumerate}
\item\label{lem:20:1} $A\oplus_1 A / X_A \rightarrow M\oplus_1 M / X_M$
is an isometry;
\item\label{lem:20:2} whenever $a\in A_0, (x,y)\in \mc G(\alpha^M_z)$
are such that $\|a-y\| + \|x\| <1$ there is
$(b,c) \in \mc G(\alpha^A_z)$ with $\|a-c\| + \|b\| <1$.
\end{enumerate}
\end{lemma}
\begin{proof}
Suppose that (\ref{lem:20:1}) holds, and that we have $a,(x,y)$ as in 
(\ref{lem:20:2}).  Then $\|(a,0) + X_M\|<1$ so by (\ref{lem:20:1}), we have that also
$\|(a,0)+X_A\|<1$, and hence there are $(b,c)\in \mc G(\alpha^A_z)$ with
$\|a-c\| + \|b\| <1$, as required.

Conversely, suppose that (\ref{lem:20:2}) holds.  As $A\oplus_1 A / X_A \rightarrow
M\oplus_1 M / X_M$ is always norm-decreasing, it follows easily that (\ref{lem:20:2})
implies that $\|(a,0)+X_M\| = \|(a,0)+X_A\|$ for $a\in A_0$.  It hence suffices to show
that $\{ (a,0)+X_A : a\in A_0 \}$ is norm dense in $A\oplus_1 A/X_A$.  Choose
$(a,b) \in A\oplus A$ and $\epsilon>0$.  There is $n$ with $\|\mc R_n(b)-b\|<\epsilon$.
Then $(\alpha_z\mc R_n(b), -\mc R_n(b)) \in X_A$ and so $(a,b) + X_A =
(a+\alpha_z\mc R_n(b), b-\mc R_n(b)) + X_A$.  As $A_0$ is dense, there is
$a_0 \in A_0$ with $\|a+\alpha_z\mc R_n(b) - a_0\|<\epsilon$.  It follows that
$\| (a,b) - (a_0,0) + X_A \| < 2\epsilon$, as required.
\end{proof}

\begin{proposition}\label{prop:21}
Let $A_0\subseteq A$ be a dense subset.  Suppose that for each $a\in A_0$ and $\epsilon>0$
there are contractive linear maps $T,S:M\rightarrow A$ with $\|S(a)-a\|<\epsilon$ and
with $(T(x),S(y))\in \mc G(\alpha^A_z)$ for each $(x,y)\in \mc G(\alpha^M_z)$.
Then the unit ball of $\mc G(\alpha^{M_*}_z)$ is weak$^*$-dense in the unit ball of
$\mc G(\alpha^{A^*}_z)$.
\end{proposition}
\begin{proof}
We verify condition (\ref{lem:20:2}) in Lemma~\ref{lem:20}.  For $a\in A_0$ and
$(x,y)\in\mc G(\alpha^M_z)$ with $\|a-y\|+\|x\|<1$, choose $\epsilon>0$, and pick
$T,S$ as in the hypothesis.  Then $(T(x),S(y))\in\mc G(\alpha^A_z)$ and
\[ \|a - S(y)\| + \|T(x)\| \leq \|a-S(a)\| + \|S(a-y)\| + \|T(x)\|
< \epsilon + \|a-y\| + \|x\|. \]
For $\epsilon>0$ sufficiently small, we have $\|a-c\|+\|b\|<1$ for $(b,c) = (T(x),S(y))
\in\mc G(\alpha^A_z)$ hence showing condition (\ref{lem:20:2}).
\end{proof}

Let us make links with the machinery developed in Section~\ref{sec:kap}.  Firstly,
another way to prove the main theorem in that section would be to use the central
projection $p\in A^{**}$ to define maps $T,S$ with the properties in
Proposition~\ref{prop:21}.  Secondly, we showed that if $M_*$ is identified with
$pA^*$, so identifying $M$ with $pA^{**}$, then $\mc G(\alpha^M_z)$ can be
identified with $(p\oplus p)\mc G(\alpha^A_z)^{\perp\perp} \subseteq
\mc G(\alpha^A_z)^{\perp\perp}$.  One can easily show that then
\[ \|(a,b)+X_M\| = \| (pb, -pa) + \mc G(\alpha^A_z)^{\perp\perp} \|, \]
the latter norm being on $A^{**}\oplus_1 A^{**}  / \mc G(\alpha^A_z)^{\perp\perp}
= (A\oplus_1 A/\mc G(\alpha^A_z))^{**}$.  Indeed, if $(z,w)\in
\mc G(\alpha^A_z)^{\perp\perp}$ with $\|pb+z\| + \|pa-w\|<1$ then
also $\|pb+pz\| + \|pa-pw\|<1$.  Then $(pz,pw)$ is identified with
$(x,y)\in\mc G(\alpha^M_z)$, so $(-y,x)\in X_M$, and $\|a-y\| + \|b+x\|<1$,
so $\|(a,b)+X_M\|<1$; and one can reverse this argument.

Note that the map $A\rightarrow A^{**}; a\mapsto pa$
is an isometry (as $A\rightarrow M$ is an isometry) but in general this is not the
canonical map $A\rightarrow A^{**}$.  Thus, showing that
$A\oplus_1A / X_A \rightarrow M\oplus_1M/X_M$ is an isometry is equivalent to showing
that $\|(a,b)+\mc G(\alpha^A_z)\| = \|(a,b)+\mc G(\alpha^A_z)^{\perp\perp}\|
= \|(pa,pb) + \mc G(\alpha^A_z)^{\perp\perp}\|$ for all $a,b\in A$.  This in turn
requires us to have knowledge of $\|(p^\perp a,p^\perp b) +
\mc G(\alpha^A_z)^{\perp\perp}\|$ where $p^\perp = 1-p$.  The link with
Proposition~\ref{prop:21} is that the maps $T,S$ there could
be assembled into a net, and then a weak$^*$-limit taken,
thus obtaining $T,S:M=pA^{**}\rightarrow A^{**}$ with $S(pa)=a$ for $a\in A$, and mapping
$\mc G(\alpha^M_z)$ to $\mc G(\alpha^A_z)^{\perp\perp}$.  We do not see a way to push
this line of argument further in general.

\subsection{Implemented automorphism groups}\label{sec:iag}

Let $M$ be a von Neumann algebra.  We recall the notion of a \emph{standard form}
for $M$, \cite{h}, \cite[Chapter~IX]{tak2}, which we shall denote by
$(M, L^2(M), J_M, L^2(M)^+)$.  By \cite[Theorem~3.2]{h} for any (weak$^*$-continuous)
automorphism $\alpha$ of $M$, there is a unique $u$, a unitary on $L^2(M)$, with
$\alpha(x) = uxu^*$ and $J_M = u J_M u^*, u(L^2(M)^+) = L^2(M)^+$.  Furthermore,
if $(\alpha_t)$ is a one-parameter automorphism group of $M$ and $(u_t)$ the resulting
unitaries, then $(u_t)$ is strongly continuous, \cite[Corollary~3.6]{h}.

The following is a generalisation of a similar result of ours, \cite[Lemma~3]{ds};
but that proof is not correct, as it requires taking a linear span.  Indeed, the
following could also be shown by adapting the (corrected) proof of \cite[Lemma~3]{ds}.

\begin{proposition}\label{prop:14}
With $M, (\alpha_t), (u_t)$ as above, consider
\[ D = \lin \{ \omega_{\xi,\eta} : \xi \in D(u_{i/2}), \eta\in D(u_{-i/2}) \}
\subseteq M_*. \]
Then $D$ is a core for $\mc G(\alpha^{M_*}_{-i/2})$.
\end{proposition}
\begin{proof}
We first note that $D\subseteq D(\alpha^{M_*}_{-i/2})$.  Indeed, if
$\xi \in D(u_{i/2}), \eta\in D(u_{-i/2})$ and $x\in D(\alpha_{-i/2})$ then by
Proposition~\ref{prop:foura} we have that $D(u_{-i/2} x u_{i/2}) = D(u_{i/2})$
and $u_{-i/2} x u_{i/2}$ is closable with closure $y = \alpha_{-i/2}(x)$.
Then
\[ \ip{(-y,x)}{(\omega_{\xi,\eta}, \omega_{u_{i/2}\xi, u_{-i/2}\eta})}
= (u_{-i/2}\eta|xu_{i/2}\xi) - (\eta|y\xi)
= (\eta|u_{-i/2}xu_{i/2}\xi) - (\eta|y\xi) = 0. \]
This shows that $(\omega_{\xi,\eta}, \omega_{u_{i/2}\xi, u_{-i/2}\eta})
\in \mc G(\alpha^{M_*}_{-i/2})$ as required.

As $u_t u_z = u_z u_t$ for any $t\in\mathbb R, z\in\mathbb C$, it follows that
$D$ is $(u_t)$-invariant.  As $D(u_{-i/2})$ and $D(u_{i/2})$ are dense in $H$,
it follows that $D$ is dense in $M_*$.
The result now follows from Theorem~\ref{thm:thm1}.
\end{proof}

It would be interesting to characterise all of $\mc G(\alpha^{M_*}_{-i/2})$ (in
a similar way) and not just a core.
As $M$ is in standard form, we know that $M_* = \{ \omega_{\xi,\eta}:
\xi,\eta\in L^2(M) \}$, with no linear span required.  It is tempting to believe
that this should allow us to improve the above result by removing the linear span.
However, the following example shows that, naively, this will not work (though in
the special setting of the proposition, the result might still hold-- we have been
unable to decide this).

\begin{example}
We construct a one-parameter isometry group $(\alpha_t)$ on $\ell^1 = (\ell^\infty)_*$
and a dense set $D\subseteq\ell^1$ which is $(\alpha_t)$-invariant, with $D\subseteq
D(\alpha_{-i})$ but such that $D' = \{ (x,\alpha_{-i}(x)) : x\in D \}$ is not dense in
$\mc G(\alpha_{-i})$.  As in Example~\ref{ex:four} we can embed this example into the
predual of an automorphism group on a von Neumann algebra.

Define $\alpha_t(x) = (e^{int} x_n)$ for $x = (x_n)_{n\geq 1}\in\ell^1$.
Thus $D(\alpha_{-i}) = \{ x=(x_n)\in\ell^1 : \sum_n e^n |x_n| < \infty \}$.
Let $(m(k))_{k\geq 1}$ be a strictly increasing sequence with $m(1)>1$.  Define $D
\subseteq\ell^1$ by saying that $x=(x_n)\in D$ if and only if $x\in D(\alpha_{-i})$ and
there exists $N$ so that $|x_1| = e^{1+m(N)}|x_{1+m(N)}|$.
Clearly $D$ is $(\alpha_t)$-invariant.

Given $(x_n)\in\ell^1$ of finite support, that is, there is $K\geq 1$ with $x_n=0$ for
$n>K$, then define $y=(y_n)$ by $y_n = x_n$ for $n\leq K$, $y_{1+m(K)} = e^{-1-m(K)}y_1$,
and $y_n=0$ otherwise.  As $(m(k))$ is strictly increasing and $m(1)>1$, we have that
$1+m(K)>K$ so $(y_n)$ is well-defined.  As $\sum_n |y_n|e^n = \sum_{n\leq K} |x_n|e^n
+ |y_1| < \infty$ we see that $y\in D$.  Clearly $\|x-y\| = e^{-1-m(K)}|x_1|$ which is
arbitrarily small (by choosing $K$ large).  We conclude that $D$ is dense in $\ell^1$.

Let $\delta_1\in\ell^1$ be the sequence which is $1$ at $1$ and $0$ otherwise.
Suppose towards a contradiction that there is $x\in D$ with $\|
(\delta_1, \alpha_{-i}(\delta_1)) - (x, \alpha_{-i}(x)) \|<\epsilon$.
Then $|x_1-1| < \epsilon$ and $\|e^1\delta_1 - \alpha_{-i}(x)\| < \epsilon$.
As $x\in D$ there is $N$ with $|x_{1+m(N)}| = e^{-1-m(N)} |x_1|$ so
$e^{1+m(N)} |x_{1+m(N)}| = |x_1| > 1-\epsilon$ showing that
$\|e^1\delta_1 - \alpha_{-i}(x)\| > 1-\epsilon$ which is a contradiction
if $\epsilon<1/2$.
\end{example}

\subsection{Tensor products}\label{sec:tp}

While not directly related to duality, we wish to briefly consider tensor products.
This was also done in \cite[Section~4]{kus1}, and so we shall just given an overview.
Our aim is to demonstrate how ``uniqueness'' results for analytic generators, compare
\cite[Proposition~F1]{mnw} or \cite[Lemma~4.4]{h1} for example, can be shown using the
smearing technique, instead of Carlson's Lemma from complex analysis.

For ease, we shall simply work with the Banach space projective tensor product,
see \cite[Chapter~2]{ryan} for example; see \cite[Section~4]{kus1} for more general
considerations.  Let $E,F$ be Banach spaces,
and $(\alpha_t), (\beta_t)$ be one-parameter isometry groups on $E,F$ respectively.
Then on the projective tensor product $E\proten F$, it is clear that $\gamma_t=
\alpha_t\otimes\beta_t$ defines a one-parameter isometry group.  By considering analytic
extensions, it follows that if $x\in D(\alpha_z), y\in D(\beta_z)$ then $x\otimes y\in
D(\gamma_z)$ with $\gamma_z(x\otimes y) = \alpha_z(x) \otimes \beta_z(y)$.  It follows
that the algebraic tensor product $D(\alpha_z)\otimes D(\beta_z)$ is a subset of
$D(\gamma_z)$.  As $D(\alpha_z)\otimes D(\beta_z)$ is $(\gamma_t)$-invariant and dense
in $E\proten F$, Theorem~\ref{thm:thm1} shows that $D(\alpha_z)\otimes D(\beta_z)$ is a
core for $\gamma_z$.

We state the following for a linear map, but there is an obvious extension
(following \cite[Proposition~F1]{mnw}) to multi-linear maps.

\begin{proposition}
Let $\theta:E\rightarrow F$ be a bounded linear map with $\theta \alpha_{-i}
\subseteq \beta_{-i}\theta$.  Then $\theta\alpha_t = \beta_t\theta$ for all
$t\in\mathbb R$.
\end{proposition}
\begin{proof}
We consider $(\beta_t^*)$ on $F^*$, so, again, $\mc G(\beta^*_{-i})
= \{ (-\lambda,\mu) : (\mu,\lambda)\in \mc G(\beta_{-i})^\perp \}$.  Let $F_0\subseteq
F^*$ be the collection of $\mu$ such that $t\mapsto \beta_t^*(\mu)$ is norm continuous.
This is readily seen to be a closed subspace, and by Proposition~\ref{prop:5}, $F_0$
is weak$^*$-dense in $F^*$.  Let $(\beta^0_t)$ be the restriction of $(\beta^*_t)$ to
$F_0$, which forms a norm continuous isometry group.

Our hypothesis is that if $x\in D(\alpha_{-i})$ then $\theta(x) \in D(\beta_{-i})$
and $\beta_{-i} \theta(x) = \theta \alpha_{-i}(x)$, that is,
$(\theta(x), \theta\alpha_{-i}(x)) \in \mc G(\beta_{-i})$.  If $\mu\in D(\beta^0_i)$
with $\lambda=\beta^0_i(\mu)$, then $(\lambda,\mu)\in \mc G(\beta^0_{-i})$ and so
$(\lambda,\mu)\in\mc G(\beta^*_{-i})$ so $(-\mu,\lambda)\in\mc G(\beta_{-i})^\perp$.
It follows that $\ip{\mu}{\theta(x)} = \ip{\lambda}{\theta\alpha_{-i}(x)}$.

Let $\gamma_t = \alpha_t \otimes \beta^0_{-t}$ on $E\proten F_0$, so that $D(\alpha_{-i})
\otimes D(\beta^0_i)$ is a core for $\gamma_{-i}$.  Define $T:E\proten F_0\rightarrow
\mathbb C$ by $T(x\otimes\mu) = \ip{\mu}{\theta(x)}$.  Then $T(x\otimes\mu) =
T(\alpha_{-i}(x)\otimes\beta^0_i(\mu))$ for all $x\in D(\alpha_{-i}),\mu\in D(\beta^0_i)$.
Thus $T(u) = T(v)$ for all $(u,v)\in\mc G(\gamma_{-i})$.  As $T\in (E\proten F_0)^*$,
this means that $(T,T)\in\mc G(\gamma^*_{-i})$.  Exactly as in Remark~\ref{rem:5},
this means we can find an entire, bounded, extension of the orbit map $t\mapsto
\gamma^*_t(T)$, and so $\gamma^*_t(T) = T$ for all $t$.

It follows that $\ip{\beta^0_{-t}(\mu)}{\theta\alpha_t(x)} = \ip{\mu}{\theta(x)}$ for
each $x\in E, \mu\in F_0$.  As $F_0$ is weak$^*$-dense in $F^*$, and $\beta^0_{-t} =
\beta^*_{-t}$ on $F_0$, it follows that $\beta_{-t}\theta\alpha_t(x)=\theta(x)$ for each
$x\in E$, that is, $\beta_t\theta =\theta\alpha_t$, as required.
\end{proof}

\begin{corollary}
Let $E=(E_*)^*, F=(F_*)^*$ be dual spaces, and $(\alpha_t), (\beta_t)$ be
weak$^*$-continuous.  If $\theta:E\rightarrow F$ is weak$^*$-continuous with
$\theta\alpha_{-i} \subseteq \beta_{-i}\theta$ then $\theta\alpha_t=\beta_t\theta$.
\end{corollary}
\begin{proof}
There is $\theta_*:F_*\rightarrow E_*$ with $\theta = (\theta_*)^*$.  That
$\theta\alpha_{-i} \subseteq \beta_{-i}\theta$ is equivalent to
$(x,y)\in\mc G(\alpha_{-i}) \implies (\theta(x), \theta(y))\in\mc G(\beta_{-i})$.
Let $(\mu,\lambda)\in\mc G(\beta^{F_*}_{-i})$ so that $(-\lambda,\mu)
\in {}^\perp\mc G(\beta_{-i})$.  Thus, for $(x,y)\in\mc G(\alpha_{-i})$, we have
that
\[ \ip{(\theta(x),\theta(y))}{(-\lambda,\mu))} = 0 \implies
\ip{x}{\theta_*(\lambda)} = \ip{y}{\theta_*(\mu)}, \]
and so $(\theta_*(\mu), \theta_*(\lambda))\in\mc G(\alpha^{E_*}_{-i})$.
We have hence shown that $\theta_*\beta^{F_*}_{-i} \subseteq \alpha^{E_*}_{-i}\theta_*$.
Hence $\theta_*\beta^{F_*}_t = \alpha^{E_*}_t \theta_*$ for all $t$, so taking adjoints
gives the required conclusion.
\end{proof}

Applied with $E=F$ and $\theta$ the identity map, this result shows that the generator
$\alpha_{-i}$ uniquely determines $(\alpha_t)$.

\section{Locally compact quantum groups}\label{sec:lcqgs}

We give a brief introduction to locally compact quantum groups, \cite{kus2, kus3,
kv, mnw, vd}.  We write $\G$ for the abstract object thought of as a locally compact
quantum group, which has a concrete operator-algebraic realisation as either the
von Neumann algebra $L^\infty(\G)$ or the $C^*$-algebra $C_0(\G)$.  We write $\Delta$
for the coproduct, either a unital normal injective $*$-homomorphism $L^\infty(\G)
\rightarrow L^\infty(\G)\vnten L^\infty(\G)$, or a non-degenerate $*$-homomorphism $C_0(\G) \rightarrow
M(C_0(\G)\otimes C_0(\G))$.  The left Haar weight, via the GNS construction, gives
rise to a Hilbert space $L^2(\G)$ on which $L^\infty(\G)$ and $C_0(\G)$ act.
We denote the dual quantum group by $\hh\G$, and identify $L^2(\G)$ with $L^2(\hh\G)$.
We recall the fundamental multiplicative unitary $W \in M(C_0(\G)\otimes C_0(\hh\G))
\subseteq L^\infty(\G)\vnten L^\infty(\hh\G)$ which implements the coproduct as
$\Delta(x) = W^*(1\otimes x)W$.  We can recover $C_0(\G)$ as the norm closure of
$\{ (\id\otimes\omega)W : \omega\in \mc B(L^2(\G))_*\}$, and similarly $C_0(\hh\G)$
as the norm closure of $\{ (\omega\otimes\id)W : \omega\in \mc B(L^2(\G))_*\}$.

We write $L^1(\G)$ for the predual of $L^\infty(\G)$, and write $M(\G)$ for the dual
of $C_0(\G)$.  These both become Banach algebras for the ``convolution product'' induced
by the coproduct.  Furthermore, $M(\G)$ is a dual Banach algebra, and the isometric
inclusion $L^1(\G)\rightarrow M(\G)$ is a homomorphism.

The group inverse operation, for a quantum group, is represented by the antipode,
which in general is an unbounded operator $S$.  Two related objects are $R$, the
unitary antipode, which is a $*$-antiautomorphism of $C_0(\G)$ which extends to
a normal map on $L^\infty(\G)$, and $(\tau_t)$ the scaling group, which is a one-parameter
automorphism group on $C_0(\G)$ which extends to a weak$^*$-continuous automorphism
group of $L^\infty(\G)$.  Thus we are in precisely the situation considered elsewhere
in this paper, and furthermore, the scaling group exactly governs the unboundedness of
the antipode, because $S = R\tau_{-i/2}$.  We recall that $R$ and $(\tau_t)$ commute,
from which it follows that $R\tau_{-i/2} = \tau_{-i/2}R$.
Let us think briefly about what exactly we mean by $S = R \tau_{-i/2}$:
\begin{itemize}
\item In \cite[Definition~5.21]{kv}, $S$ is defined to be $R\tau_{-i/2}$,
here acting on $C_0(\G)$.  As we are considering norm-continuous $(\tau_t)$ there is
essentially no risk of ambiguity.
\item In \cite[Page~74]{kus3}, $S$ is defined to be $R\tau_{-i/2}$, and it is not
entirely clear what is meant by $\tau_{-i/2}$.  Part of our motivation for developing
the material in Sections~\ref{sec:ag} and~\ref{sec:duality} was to show that, actually,
the particular definition of $\tau_{-i/2}$ is unimportant.
\item In \cite[Definition~2.23]{vd}, $S$ is defined to be $R\tau_{-i/2}$.
This paper takes \emph{as definition} that $\tau_{-i/2}$ is the adjoint of
$\tau_{*,-i/2}$ where $(\tau_{*,t})$ is the one-parameter isometry group on $L^1(\G)$.
Of course, by Theorem~\ref{thm:two}, this agrees with the usual meaning of $\tau_{-i/2}$.
\end{itemize}

As $S = R\tau_{-i/2}$ and $R$ and $(\tau_t)$ commute, it follows that
$D(S) = D(\tau_{-i/2})$.  
As the inclusion $C_0(\G) \rightarrow L^\infty(\G)$ intertwines $R$,
it follows easily that questions about
$S$ can immediately be reduced to questions about $\tau_{-i/2}$.  For example, the
following is immediate from Theorem~\ref{thm:six}.

\begin{theorem}
The unit ball of $\mc G(S) \subseteq C_0(\G)\oplus_\infty C_0(\G)$ is
weak$^*$-dense in the unit ball of 
$\mc G(S) \subseteq L^\infty(\G)\oplus_\infty L^\infty(\G)$.
\end{theorem}

As $\Delta \tau_t = (\tau_t\otimes\tau_t)\Delta$ it follows that $(\tau^{L^1(\G)}_t)$
is an automorphism group for $L^1(\G)$, and similarly $(\tau^{M(\G)}_t)$ is a
weak$^*$-continuous automorphism group for $M(\G)$.  The natural way to induce an
involution on $L^1(\G)$ is to use the antipode and the involution on $L^\infty(\G)$,
leading to definition of $L^1_\sharp(\G)$ as those $\omega\in L^1(\G)$ such that
there is $\omega^\sharp\in L^1(\G)$ with $\overline{\ip{S(x)^*}{\omega}} =
\ip{x}{\omega^\sharp}$ for all $x\in D(S)$.  The map $\omega\mapsto\omega^\sharp$
becomes an involution on $L^1_\sharp(\G)$.
It is shown in \cite[Proposition~3.1]{kus4} that then $L^1_\sharp(\G)\rightarrow
C_0(\hh\G); \omega\mapsto (\omega\otimes\id)(W)$ is a $*$-homomorphism.

The following is the natural extension of this definition to $M(\G)$.

\begin{definition}
We define $M_\sharp(\G)$ to be the collection of $\mu\in M(\G)$ such that there is
$\mu^\sharp\in M(\G)$ with $\overline{\ip{\mu}{S(a)^*}} = \ip{\mu^\sharp}{a}$ for
$a\in D(S)$.
\end{definition}

For $\mu\in M(\G)$ we write $\mu^*\in M(\G)$ for the functional $a\mapsto
\overline{\ip{\mu}{a^*}}$.
Given $\mu\in M_\sharp(\G)$ define $\lambda = R^*(\mu^*)$.
For $a\in D(\tau_{-i/2})$ we have
\[ \ip{\mu^\sharp}{a} =
\overline{\ip{\mu}{S(a)^*}} = \ip{\mu^*}{S(a)} = \ip{\lambda}{\tau_{-i/2}(a)}. \]
Thus $(\mu^\sharp, -\lambda) \in \mc G(\tau_{-i/2})^\perp$ so
$(\lambda, \mu^\sharp) \in \mc G(\tau_{-i/2}^{M(\G)})$.  We can reverse this argument,
thus establishing that $\mu\in M_\sharp(\G)$ if and only if $R^*(\mu^*)\in
D(\tau_{-i/2}^{M(\G)})$ and then $\mu^\sharp = \tau_{-i/2}^{M(\G)}(R^*(\mu^*))$.
An analogous argument holds for $L^1_\sharp(\G)$.

The definitions of both $L^1_\sharp(\G)$ and $M_\sharp(\G)$ are both ``graph like'', in
that given, say, $\mu\in M(\G)$, we have that $\mu\in M_\sharp(\G)$ when there exists
$\mu^\sharp\in M(\G)$ with a certain property.  The next two results show that we can
instead impose conditions purely on $\mu$.  The first result is an application of
Hahn-Banach, but the second result is somewhat less obvious.

\begin{proposition}
Let $\mu\in M(\G)$ be such that $\mu^*\circ S: D(S)\subseteq C_0(\G) \rightarrow
\mathbb C$ is bounded.  Then $\mu\in M_\sharp(\G)$.
\end{proposition}
\begin{proof}
Let $\mu_0 = \mu^*\circ S$ and take a Hahn-Banach extension to an element
$\mu^\sharp \in M(\G)$ (or simply extend by continuity, as $D(S)$ is dense in $C_0(\G)$).
Thus, for $a\in D(S)$,
\[ \ip{\mu^\sharp}{a} = \ip{\mu^*}{S(a)} = \overline{\ip{\mu}{S(a)^*}}, \]
so by definition, $\mu\in M_\sharp(\G)$.
\end{proof}

\begin{theorem}\label{thm:10}
Let $\omega\in L^1(\G)$ be such that either:
\begin{enumerate}
\item\label{thm:10:1}
$\omega^*\circ S: D(S)\subseteq C_0(\G) \rightarrow \mathbb C$ is bounded; or
\item\label{thm:10:2}
$\omega^*\circ S: D(S)\subseteq L^\infty(\G) \rightarrow \mathbb C$ is bounded.
\end{enumerate}
Then $\omega\in L^1_\sharp(\G)$.
\end{theorem}
\begin{proof}
Let us write $S_0 = R_0 \circ
\tau^{0}_{-i/2}$ for $S$ on $C_0(\G)$, and $S_\infty = R_\infty \circ
\tau^\infty_{-i/2}$ for $S$ on $L^\infty(\G)$.  As the inclusion
$C_0(\G)\rightarrow L^\infty(\G)$ intertwines $R_0$ and $R_\infty$, and intertwines
$\tau^0_{-i/2}$ and $\tau^\infty_{-i/2}$, we see that it intertwines $S_0$ and $S_\infty$.

If (\ref{thm:10:1}) holds, then by
the previous proposition, $\omega\in M_\sharp(\G)$, that is, $R_0^*(\omega^*)
\in D(\tau_{-i/2}^{M(\G)})$.  We then apply Theorem~\ref{thm:seven} to see that
$R_0^*(\omega^*) \in D(\tau_{-i/2}^{L^1(\G)})$.  However, $R_0^*(\omega^*)$ is
equal to the image of $(R_\infty)_*(\omega^*) \in L^1(\G)$ in $M(\G)$,
and so $\omega\in L^1_\sharp(\G)$, as required.

Now suppose that (\ref{thm:10:2}) holds.  Then the composition
$D(S_0) \rightarrow D(S_\infty) \rightarrow
\mathbb C$ is bounded, that is, (\ref{thm:10:1}) holds.  The claim follows.
\end{proof}

In \cite[Proposition~14]{ds} the author and Salmi showed the following, via
``Banach algebraic'' techniques.  We wish here to quickly record how to use the more
abstract approach of Section~\ref{sec:kap_dual}.  We recall that $\G$ is coamenable
when $L^1(\G)$ has a bounded approximate identity, \cite{bt}.

\begin{proposition}\label{prop:15}
Let $\G$ be coamenable.  For any $\mu\in D(\tau^{M(\G)}_z)$ there is a net $(\omega_i)$
in $D(\tau^{L^1(\G)}_z)$ with $\omega_i\rightarrow\mu$ weak$^*$ and with
$\|\omega_i\|\leq\|\mu\|$ and $\|\tau^{L^1(\G)}_z(\omega_i)\|\leq\|\tau^{M(\G)}_z(\mu)\|$
for each $i$.
\end{proposition}
\begin{proof}
We will use Proposition~\ref{prop:21}.  Let $A_0 = \{ (\id\otimes\phi)(W) :
\phi\in\mc B(L^2(\G))_* \} \subseteq C_0(\G)$ a dense subset (actually, subalgebra).

For $\omega\in L^1(\G)$ consider the map $P_\omega : L^\infty(\G)\rightarrow
L^\infty(\G)$ given by $P_\omega(x) = (\id\otimes\omega)\Delta(x)$.  This actually
maps into $C^b(\G)$, see for example the proof of \cite[Theorem~2.4]{run1}.
Let $(x,y)\in\mc G(\tau^{L^\infty(\G)}_z)$ and $(\omega,\phi), (\alpha,\beta)\in
\mc G(\tau^{L^1(\G)}_z)$.  Then
\begin{align*} \ip{(P_\phi(x), P_\omega(y))}{(-\beta,\alpha)}
&= \ip{(\id\otimes\omega)\Delta(y)}{\alpha} - \ip{(\id\otimes\phi)\Delta(x)}{\beta} \\
&= \ip{y}{\alpha\omega} - \ip{x}{\beta\phi}
= \ip{(y,-x)}{(\alpha\omega, \beta\phi)} = 0,
\end{align*}
because $(\alpha\omega, \beta\phi) \in \mc G(\tau^{L^1(\G)}_z)$ by
Proposition~\ref{prop:one}.  As $(\alpha,\beta)$ was arbitrary, this shows that
$(P_\phi(x), P_\omega(y)) \in \mc G(\tau^{L^\infty(\G)}_z)$.

As $\G$ is coamenable, $L^1(\G)$ has a contractive approximate identity, and so by
Proposition~\ref{prop:19}, also $\mc G(\tau^{L^1(\G)}_z)$ has a contractive approximate
identity, say $(\omega_i, \phi_i)$.
For the moment, suppose that $\G$ is compact, so $C_0(\G) = C^b(\G) = C(\G)$.
For each $i$, the pair $(P_{\phi_i}, P_{\omega_i})$ are contractive maps $L^\infty(\G)
\rightarrow C(\G)$ which map $\mc G(\tau^{L^\infty(\G)}_z)$ to
$\mc G(\tau^{L^\infty(\G)}_z) \cap (C_0(\G)\oplus C_0(\G)) = 
\mc G(\tau^{C(\G)}_z)$, by Proposition~\ref{prop:4}.  
To invoke Proposition~\ref{prop:21} it hence remains to show that $P_{\phi_i}(a)
\rightarrow a$ in norm, for each $a = (\id\otimes\phi)(W) \in A_0$.  However, then
\begin{align*}
P_{\phi_i}(a) &= (\id\otimes\phi_i)\Delta\big( (\id\otimes\phi)(W) \big)
= (\id\otimes\phi_i\otimes\phi)(W_{13}W_{23}) \\
&= (\id\otimes\phi)(W(1\otimes\lambda(\phi_i)))
= (\id\otimes\lambda(\phi_i)\phi)(W)
\end{align*}
where $\lambda(\phi_i) = (\phi_i\otimes\id)(W) \in C_0(\hh\G)$.
As $(\phi_i)$ is a bounded approximate identity (bai) in $L^1(\G)$, it follows that
$\lambda(\phi_i)$ is a bai for $C_0(\hh\G)$, as
$\lambda(L^1(\G))$ is dense in $C_0(\hh\G)$.  For any bai $(\hh a_i)$ in $C_0(\hh\G)$
where have that $\hh a_i \phi \rightarrow \phi$ in norm, for $\phi\in L^1(\G)$.
Thus $P_{\phi_i}(a) \rightarrow a$ in norm, as required.

To deal with the non-compact case, we apply Proposition~\ref{prop:19} to find
a contractive approximate identity $(e_i,f_i)$ in $\mc G(\tau^{C_0(\G)}_z)$.
Then, for any $i,j$, we have that $x\mapsto e_i P_{\phi_j}(x)$ maps $L^\infty(\G)$
to $C_0(\G)$, and again for $a\in A_0$, for sufficiently large $i,j$ we have that
$e_i P_{\phi_j}(a)$ is close to $a$.  The proof now follows as before.
\end{proof}

Of course, if for example $(\tau_t)$ is trivial (for example, $\G$ is a Kac algebra)
then we certainly do not need $\G$ to be coamenable.  In this case, the conditions
of Lemma~\ref{lem:20} follow immediately from the triangle-inequality.  We continue to
wonder if the result above is really true for any $\G$.

We know that the left Haar weight $\varphi$ is relatively invariant under $(\tau_t)$,
that is, there is $\nu>0$, the scaling constant, such that $\varphi(\tau_t(x)) =
\nu^{-t} \varphi(x)$ for $x\in L^\infty(\G)_+$.  Denote $\mf n_\varphi = \{ x\in
L^\infty(\G) : \varphi(x^*x)<\infty\}$ and let $\Lambda:\mf n_\varphi\rightarrow L^2(\G)$
be the GNS map.  We may hence define a one-parameter unitary group $(P^{it})$ on
$L^2(\G)$ by $P^{it}\Lambda(x) = \nu^{t/2} \Lambda(\tau_t(x))$ for $x\in\mf n_\varphi,
t\in\mathbb R$.  Then $\tau_t(x) = P^{it} x P^{-it}$ and so we are in the setting
of Section~\ref{sec:iag}.  We remark that one can easily adapt the proof of
\cite[Proposition~3.7]{h} to show that $(P^{it})$ is the canonical unitary
implementation of $(\tau_t)$, in the sense of Section~\ref{sec:iag}.  The following
is now immediate from Proposition~\ref{prop:14}, and corrects \cite[Lemma~3]{ds}
by requiring the linear span.

\begin{proposition}
The set $D = \lin\{ \omega_{\xi,\eta} : \xi\in D(P^{1/2}), \eta\in D(P^{-1/2}) \}$
is a core for $L^1_\sharp(\G)$.
\end{proposition}

We finish with an application of Proposition~\ref{prop:13}.  We recall that $\G$
is \emph{amenable} when there is a state $m\in L^\infty(\G)^*$ with
$\ip{m}{(\id\otimes\omega)\Delta(x)} = \ip{m}{x} \ip{1}{\omega}$ for $\omega\in L^1(\G),
x\in L^\infty(\G)$, see \cite{bt}.  From \cite[Theorem~3.2]{bt} we know that if $\hh\G$
is coamenable then $\G$ is amenable; the converse is a well-known open question.
By \cite[Theorem~3.3]{bt} we know that when $\G$ is amenable, $L^\infty(\hh\G)$
is injective, \cite[Chapter~XVI]{tak3}, that is, there is a contractive projection
$\mc B(L^2(\G)) \rightarrow L^\infty(\hh\G)$.

Define $\tau_t^{\mc B}(x) = P^{it} x P^{-it}$ for $x\in\mc B(L^2(\G))$, which gives
a weak$^*$-continuous automorphism group.
The pre-adjoint of $(\tau^{\mc B}_t)$ gives a norm-continuous one-parameter isometry
group $(\tau^{\mc B_*}_t)$ on $\mc B(L^2(\G))_*$.  Let $K$ be the kernel of the quotient
$\mc B(L^2(\G))_* \rightarrow L^1(\G)$, so that $K^\perp = L^\infty(\G)$ and hence
$K = {}^\perp L^\infty(\G)$ is $(\tau^{\mc B_*}_t)$-invariant.  Thus we are in the
setting of Section~\ref{sec:quot}.  Notice that $(\tau_t^{\mc B_*/K})$ is simply
$(\tau_t^{L^1(\G)})$.  Thus Proposition~\ref{prop:13} shows
that $\mc G(\tau^{\mc B_*}_z) \rightarrow \mc G(\tau^{L^1(\G)}_z)$ is a quotient map, in
the case when $\hh\G$ is amenable (in particular, when $\G$ is coamenable).
This result is interesting, as it parallels the quotient map $\mc B(L^2(\G))_*
\rightarrow L^1(\G)$; we wonder if there is some analogue of a ``standard form'' for
$\mc G(\tau_z^{L^1(\G)})$, compare the comments in Section~\ref{sec:graph_ba}.  We again
remark that if $(\tau_t)$ is trivial, then of course $\mc G(\tau^{\mc B_*}_z) \rightarrow
\mc G(\tau^{L^1(\G)}_z)$ is a quotient map, without any amenability condition.

\appendix
\section{Commentary on the use of the Three--Lines Theorem}

The following is an addendum which has been submitted to the New York Journal of Mathematics.  We thank Stefaan Vaes for valuable comments, and for questions which brought these issues to our attention.

These comments relate to the use of the Three-Lines Theorem in Section~2.  The Three-Lines Theorem says that given a scalar-valued, \emph{bounded}, regular function $f\colon S(z)\to\mathbb C$, the function $M(y) = \sup_t |f(t+iy)|$ satisfies that $\log M$ is convex.  In particular,
\[ |f(w)| \leq \max\Big( \sup_t |f(t)|, \sup_t |f(t+z)| \Big) \qquad (w\in S(z)). \]

\subsection{Remark~2.4}\label{add:1}

Given a norm regular, bounded function $f\colon S(z) \to E$ with values in a Banach space $E$, applying the Three-Lines Theorem to functions of the form $\mu\circ f$, as $\mu$ varies over the unit ball of $E^*$, gives a vector-valued version of the Three-Lines Theorem.  We tacitly used this in Remark~2.4, but failed to note that the weak$^*$-case is more subtle.

Indeed, let $(\alpha_t)$ be a weak$^*$-continuous isometry group on $E$, a dual Banach space with predual $E_*$, again suppose $s>0$, and let $x\in D(\alpha_{is})$.  Then $\|\alpha_t(x)\| = \|x\|$ and $\|\alpha_{t+is}(x)\| = \|\alpha_{is}(x)\|$ for each $t\in\mathbb R$.  However, a priori, we do not know that $[0,s] \to [0,\infty), r \mapsto \| \alpha_{ir}(x) \|$ is bounded, as this map is not assumed to be continuous in the weak$^*$-case.

A standard corollary of the Uniform Boundedness Theorem, \cite[Chapter~III, Corollary~14.4]{Conway_FunctionalAnalysisBook}, says that $X\subseteq E$ is bounded if (and only if) for each $\mu\in E_*$, we have that $\sup\{ |\mu(x)| : x\in X \} < \infty$.  We apply this with $X = \{ \alpha_{ir}(x) : 0\leq r\leq s\}$, as for each $\mu\in E_*$, the map $[0,s]\to \mathbb C; r\mapsto \mu(\alpha_{ir}(x))$ is continuous, and so bounded.  Using that $\|\alpha_{t+ir}(x)\| = \|\alpha_{ir}(x)\|$ for $t\in\mathbb R, 0\leq r\leq s$, it follows that $S(z)\to E; w\mapsto \alpha_w(x)$ is norm bounded, and we can proceed to apply the Three-Lines Theorem to functions of the form $w\mapsto \mu(\alpha_w(x))$, as $\mu$ varies over the unit ball of $E_*$.

\subsection{Lemma~2.14}\label{add:2}

A more serious problem applies to Lemma~2.14, as here we apply the Three-Lines Theorem to a semi-norm, and it is far from clear what this actually means.

We shall repair the proof, at least in the situation we are interested in, namely for the $\sigma$-strong$^*$ topology on a von Neumann algebra $M$.  Similar remarks would apply to the $\sigma$-strong topology.  Recall that the $\sigma$-strong$^*$ topology is defined by seminorms of the form $p'_\omega(x) = \omega(xx^* + x^*x)^{1/2}$, for $\omega \in M_*^+$.
We recall that the space of linear functionals continuous for the strong$^*$-topology is precisely the predual $M_*$, see \cite[Chapter~II, Theorem~2.6]{tak1} for example.  Further, in a general locally convex space, a linear functional $\mu$ is continuous if and only if there are seminorms $p_1,\cdots,p_n$ and $\alpha_1,\cdots,\alpha_n \in [0,\infty)$, with $|\mu(x)| \leq \sum \alpha_i p_i(x)$ for all $x$, see \cite[Chapter~IV, Theorem~3.1]{Conway_FunctionalAnalysisBook} for example.

\begin{lemma}
Let $z\in\mathbb C$, and let $g \colon S(z) \to M$ be bounded and weak$^*$-regular.  Let $p = p'_\omega$ for some $\omega\in M_*^+$, let $\epsilon>0$, and suppose that $p(g(t)) \leq \epsilon$ and $p(g(z+t)) \leq \epsilon$ for each $t\in\mathbb R$.  Then $p(g(w)) \leq \epsilon$ for each $w\in S(z)$.
\end{lemma}
\begin{proof}
Choose $w\in S(z)$ and set $x = g(w) \in M$.  On the one-dimensional subspace $\mathbb Cx$ define $\mu(\lambda x) = \lambda p(x)$ for $\lambda\in\mathbb C$, so $\mu$ is linear, and $|\mu(y)| \leq p(y)$ for each $y$ in this subspace.  By Hahn-Banach, for example \cite[Chapter~III, Corollary~6.4]{Conway_FunctionalAnalysisBook}, we can extend $\mu$ to all of $M$ with $|\mu(y)| \leq p(y)$ for all $y\in M$.  Thus $\mu$ is continuous for the strong$^*$-topology, and so $\mu\in M_*$.  By hypothesis, the map $\mu\circ g\colon S(z)\to\mathbb C$ is bounded and regular, and so the Three-Lines Theorem applies to show that
\begin{align*}
|\mu(g(w))| &\leq \max\Big( \sup_t |\mu(g(t))|, |\mu(g(t+z))| \Big) \\
&\leq \max\Big( \sup_t p(g(t)), p(g(t+z)) \Big)
\leq \epsilon. \end{align*}
Thus $p(g(w)) = p(x) = \mu(x) = \mu(g(w)) \leq \epsilon$ as required.
\end{proof}

This is enough to repair the proof of Lemma~2.14 at least in the case of the strong$^*$-topology, as, in the notation of that lemma, it shows that having $p(g_n-g)\to 0$ uniformly on $\mathbb R$ and $\mathbb R+z$ is enough to give uniform convergence on all of $S(z)$.

\bigskip

\noindent\emph{Author's Address:}

\noindent
Matthew Daws, School of Mathematical Sciences,
Lancaster University,
Lancaster,
LA1 4YF,
United Kingdom

\noindent\emph{Email:} \texttt{matt.daws@cantab.net}

\end{document}